\documentclass[letterpaper, 12pt]{article}

\pdfoutput=1 

\usepackage{fullpage}
\usepackage{graphicx}
\usepackage{latexsym}
\usepackage{amsfonts,amsmath,amssymb,amsthm,bm} 

\usepackage[pdftex,bookmarks,colorlinks,breaklinks]{hyperref}  
\usepackage{color}
\definecolor{dullmagenta}{rgb}{0.4,0,0.4}   
\definecolor{darkblue}{rgb}{0,0,0.4}
\hypersetup{linkcolor=red,citecolor=blue,filecolor=dullmagenta,urlcolor=darkblue} 

\newcommand{\ie}{{\it i.e.}}
\newcommand{\Res}{{\rm Res}}

\newcommand{\x}{\mathbf{x}}
\newcommand{\y}{\mathbf{y}}
\newcommand{\ud}{\,\mathrm{d}}
\newcommand{\calA}{\mathcal{A}}
\renewcommand{\imath}{i}
\newcommand{\warrow}{\rightharpoonup}

\newtheorem{thm}{Theorem}[section]

\newtheorem{prop}[thm]{Proposition}
\newtheorem{lem}[thm]{Lemma}
\theoremstyle{definition}
\newtheorem{defn}{Definition}[section]
\theoremstyle{remark}
\newtheorem{rem}[thm]{\bf Remark}

\newtheorem{conj}{\bf Conjecture}[section]

\graphicspath{{arXiv/}} 

\title{Long-lived Scattering Resonances  and Bragg Structures}
\author{Braxton Osting and Michael I. Weinstein}
\date{\today}
 
\begin{document}
\maketitle


\begin{abstract}
We consider a system governed by the wave equation with index of refraction $n(\x)$, taken to be variable within a bounded region $\Omega\subset \mathbb R^d$, and constant in $\mathbb R^d \setminus \Omega$.  The solution of  the time-dependent wave equation with initial data, which is localized  in $\Omega$,  spreads and decays with advancing time. This rate of decay can be measured (for $d=1,3$, and more generally, $d$ odd) in terms of the eigenvalues of the scattering resonance problem, a non-selfadjoint eigenvalue problem governing the time-harmonic solutions of the wave (Helmholtz) equation which are {\it outgoing} at $\infty$. Specifically, the rate of energy escape from $\Omega$ is governed by the complex scattering eigenfrequency, which is closest to the real axis. We study the structural design problem: Find a refractive index profile $n_\star(\x)$ within an admissible class which has a scattering frequency with minimal imaginary part. The admissible class is defined in terms of the compact support of $n(\x)-1$ and pointwise upper and lower (material) bounds on $n(\x)$ for $\x \in \Omega$, \ie, $0 < n_- \le n(\x) \le n_+ < \infty$. We formulate this problem as a constrained optimization problem and prove that an optimal structure, $n_\star(\x)$ exists. Furthermore, $n_\star(\x)$ is piecewise constant and achieves the material bounds, \ie,  $n_\star(\x)\in\{n_-, n_+ \} $. In one dimension, we establish a connection between $n_\star(x)$ and the well-known class of  {\it Bragg structures},   where  $n(x)$ is constant on intervals whose length is one-quarter of the effective wavelength.
\end{abstract}

\section{Introduction and overview}
\label{sec:introduction}

Many device applications, ranging from photonic to micro-mechanical require the controlled localization of energy within a compact region of space or ``cavity''.
In such settings, an important   performance-limiting loss mechanism is {\it scattering loss}, leakage from or tunneling 
 out of the structure. We have in mind applications to wave phenomena in non-dissipative media governed by time-dependent wave equations arising, for example, in (i) electromagnetic waves in dielectric media, (ii) acoustic waves, and (iii) elastic waves. An important class of motivating examples concerns the control of light via micro- and nano-scale photonic crystal devices. For example,  see,  \cite{pc-book,BVLMTW07}.  
 
Thus, the  following optimization problem naturally arises:

 \noindent {\it Given constraints on  material parameters
  and the size of the structure surrounding the cavity, how does one design a structure which maximizes the confinement time of energy?}\

We next explain how the confinement-time of energy in a cavity can be expressed in terms of the imaginary parts of complex eigenvalue of the non-selfadjoint  scattering resonance problem (SRP). We then formulate the optimization problem, summarize the results of this paper, and review related work.

\subsection{Energy escape and the scattering resonance problem}
Our point of departure is the time-dependent wave equation for an inhomogeneous medium:
 \begin{align}
\label{eqn:wave}
n^2(\x) \ \partial_t^2 v(\x,t) = \Delta  v(\x,t) &  \quad \quad  \x \in \mathbb R^d. 
\end{align}
Here, $n(\x)$ denotes a spatially varying index of refraction,
\footnote{ The wave equation in an inhomogeneous medium with position-dependent propagation speed, $c(\x)$, is $\left( c^{-2}(\x)\partial_t^2-\Delta\right)v(\x,t)=0$. Let $c_0$ be the ``background" homogeneous medium wave speed. The index of refraction, $n$,  is defined by $n(\x)=c_0/c(\x)$. Working in non-dimensionalized  units where $c_0=1$, the wave equation becomes \eqref{eqn:wave}. 
%
} 
which we assume to satisfy upper and lower bounds: $ 0\ <\ n_- \le \ n(\x) \ \le n_+\ <\ \infty $.
We consider  structures, which are supported in a fixed compact set, \ie,
$\textrm{supp}\left(n(\x)-1\right)\ =\ \overline{\Omega},
$
where $\Omega$ is a bounded open subset of $\mathbb{R}^d$.

Solutions to the Cauchy problem for the wave equation \eqref{eqn:wave} with localized initial data conserve the energy: 
\begin{equation}
E\left[v(\cdot,t),\partial_tv(\cdot,t)\right] :=  \int_{\mathbb{R}^d} n^2(\x) |\partial_tv(\x,t)|^2 + |\nabla v(\x,t)|^2 \ud \x.
\label{conservation}
\end{equation}
Yet, such solutions 
 decay to zero as $t\to\infty$ in the pointwise or {\it 
 local energy} sense:
 \begin{equation*}
 \textrm{for any compact subset}\ K\subset\mathbb{R}^d,\ \ \int_K\ |v(\x,t)|^2\ d\x\ \to\ 0,\ \  \ {\rm as}\ \ t\to\infty.
 \end{equation*}
The rate of local energy decay or transiency of energy in a bounded set \cite{LaxPhillips} can be derived by studying the solution of the initial value problem, expressed in terms of an inverse Laplace transform of the form 
\begin{equation*}
v(\x,t)\ \sim\ \int_{i\kappa-\infty}^{i\kappa+\infty}\ e^{-i \omega t}\ (\Delta + n^2 \omega^2)^{-1}\ d\omega  \circ  v_0,   \quad  \kappa>0,
\end{equation*}
where $v_0$ is determined by the Cauchy data at $t=0$.
 The resolvent kernel, $(\Delta + n^2 \omega^2)^{-1}(\x,\y)$,  has no poles in the upper half plane.
In  spatial dimensions $d=1,3$,  it has a meromorphic continuation to the lower half plane, with only pole singularities. In spatial dimension $d=2$, the resolvent kernel has a branch cut \cite{Melrose1995}. 
These poles are called {\it scattering resonances}, {\it scattering frequencies}, or {\it scattering poles}. A corresponding solution is  referred to as a scattering resonance mode. Other terms used are: quasi-normal mode, quasi-mode, or quasi-bound state.

Due to the time-dependence $e^{-i\omega t}$ in the inverse Laplace transform representation of the solution of the time-dependent initial value problem,  time decay can be shown by deforming the contour into the lower half plane to a parallel contour along which the imaginary part is slightly larger than that of the  scattering resonance which is closest to the real  $\omega$-axis, {\ie}, the pole, 
 $\omega_\star[n]$,  that is closest to the real $\omega$-axis gives rise  to the exponential decay rate $\sim \exp\left(-\left|\Im\omega_\star[n] \right| t\right)$.  
 See, for example,  \cite{LaxPhillips,Tang:2000kb,Dolph-Cho:80,Lenoir-etal:92,Brill-Gaunaurd:87} for detailed discussions of the role of scattering resonances and their characterizations. 
  $\left|\Im \omega_\star[n]\right|$ is called the \emph{width} of the resonance and  $\tau := |\Im \omega_\star|^{-1}$ is called its \emph{lifetime}. 
\bigskip

To give a precise definition of scattering resonance solutions of the wave equation
 \eqref{eqn:wave} associated with the structure $n(\x)$, for which $n^2(\x)-1$ has compact support, we first introduce the  free space $d$-dimensional  outgoing Green's function with pole at $\x$:
\begin{equation}
\label{eqn:explicitG-1}
G(|\x-\y|,\omega) = \begin{cases}
-(2 i \omega)^{-1} \exp(i \omega |x-y|) & d=1 \\
 -(4 i )^{-1} H_0^{(1)}(\omega |\x-\y|)& d=2 \\
(4\pi|\x-\y|)^{-1} \exp(i \omega |\x-\y|) & d=3. 
\end{cases}
\end{equation}
$G(|\x-\y|,\omega)$ satisfies $\left(-\Delta_\y-\omega^2\right)G(|\x-\y|,\omega)=\delta(\x-\y)$ and, for $\omega$ real, is outwardly radiating.

A scattering resonance solution is a solution of the wave equation of the form $u(\x;\omega)e^{-i\omega t}$ which is {\it outgoing}.  In particular, $u(\x;\omega)$ satisfies the {\it Helmholtz equation}:
\begin{equation}
\left(\ \Delta\ +\ \omega^2 n^2(\x)\ \right) u(\x;\omega)\ =\ 0
\label{helm}\end{equation}
Formally, writing $n^2(\x)$ as $n^2(\x)=1 + \left(n^2(\x)-1\right)$ and apply the outgoing Green's operator to \eqref{helm} yields the equation:
\begin{align}
\label{SRP}
u(\x;\omega) =  \omega^2 \int_{\Omega} G(|\x-\y|,\omega) \ [n^2(\y) -1 ] \ u(\y;\omega ) \ud \y\ .
\end{align}
A locally integrable function $u(\x,\omega)$ which solves \eqref{SRP} is a weak solution of the Helmholtz equation, \eqref{helm}. Note that Equation \eqref{SRP} need only be solved for $\x\in \Omega$. For  $\x\notin \Omega$, \eqref{SRP} gives an explicit expression for $u(\x)$ in terms of $u(\x)$ for $\x\in \Omega$.
We shall be particularly interested in bounded and piecewise constant $n(\x)$. In this case, a solution of \eqref{SRP} is at least $C^1(\mathbb{R}^d)\cap H^2_{\rm loc}(\mathbb{R}^d)$. The outgoing condition is encoded in  $u(\x;\omega)$ being in the range of the outgoing Green's operator, a consequence of \eqref{SRP}.

\medskip

\begin{defn}\label{res-def}
\begin{enumerate}
\item We refer to the integral equation \eqref{SRP} as the scattering resonance problem 
 for the structure, $n(\x)$, (SRP).\
\item The pair $\left(\omega,u(\x;\omega)\right)$ is a scattering resonance pair if 
$\omega\in\mathbb{C}$ and $u(\x;\omega)$  is a non-trivial $L^1_{\rm loc}$ solution of \eqref{SRP}.   $u(\x,\omega)$ is called a {\it scattering resonance mode}  with  corresponding to a scattering frequency $\omega$. 
\item The set of scattering resonances for the structure $n(\x)$ is denoted $\Res_n$: 
\begin{equation}
\label{eqn:allRes} 
\Res_n := \{ \text{the set of all (complex) eigenvalues},\  \omega, \text{ of \eqref{SRP} for  index} \  n(\x) \}
\end{equation}
\end{enumerate}
\end{defn}
\medskip


The set of scattering resonance frequencies, ${\rm Res}_{n}$,  is discrete 
 and lies in the open lower half plane,  $\Im\omega<0$.
If $(\omega,u)$ is a scattering resonance pair, then so is $(-\overline{\omega}, \overline{u})$; the set $\Res_n$ is is symmetric about the imaginary axis.
 The set ${\rm Res}_{n}$ may be empty, as in the case where $n(\x)\equiv1$ or may be non-empty, as in the explicit dimension d=1,2,3 examples in Appendix \ref{sec:simple-example}. In the examples presented in dimensions $2$ and $3$, $n^2(\x)$ is radially symmetric. The SRP therefore breaks into independent radially symmetric SRP's corresponding to the independent  spherical harmonics.
 The scattering resonance frequency is independent of the particular spherical harmonic and thus we see from such examples  that resonances may have multiplicity larger than one.
 \medskip

%
%

In this paper, we study the problem of designing a refractive index profile, $n(\x)$, subject to physically motivated constraints, for which there are very long-lived resonances. By the previous discussion, this corresponds to choosing $n(\x)$ so that there are scattering resonances very close to the real axis, \ie,  small width, $|\Im \omega|$. \medskip
 
 Roughly speaking, long-lived resonances can arise in the following ways:\medskip
 
\begin{enumerate}
\item[(A)] {\bf Total internal reflection:}\ Confinement of energy can be achieved by the mechanism of (nearly) total internal reflection. Consider a spherical region in 2 or 3 spatial dimensions on which $n(\x)>1$ is constant. If the ``angular momentum'' of the resonance mode is large, the mode will be strongly confined to the interface of the cavity. In the geometric optics approximation, the light rays have very shallow angle of incidence and therefore are nearly totally internally reflected. Such modes are referred to as  {\it whispering gallery} or {\it glancing} modes and are the basis for spherical resonators; see section \ref{sec:simple-example}.
\item[(B)] {\bf Interference effects:}\ The cavity can be surrounded by strongly reflective medium which is periodic of an appropriate period. In this case, wave interference effects provide the localizing mechanism. This is the basis for the Bragg resonator or Fabry-P\'erot cavity \cite{pc-book}.
\end{enumerate}
 
   \begin{figure}[t]
 \begin{center}
 \includegraphics[width=2.5in]{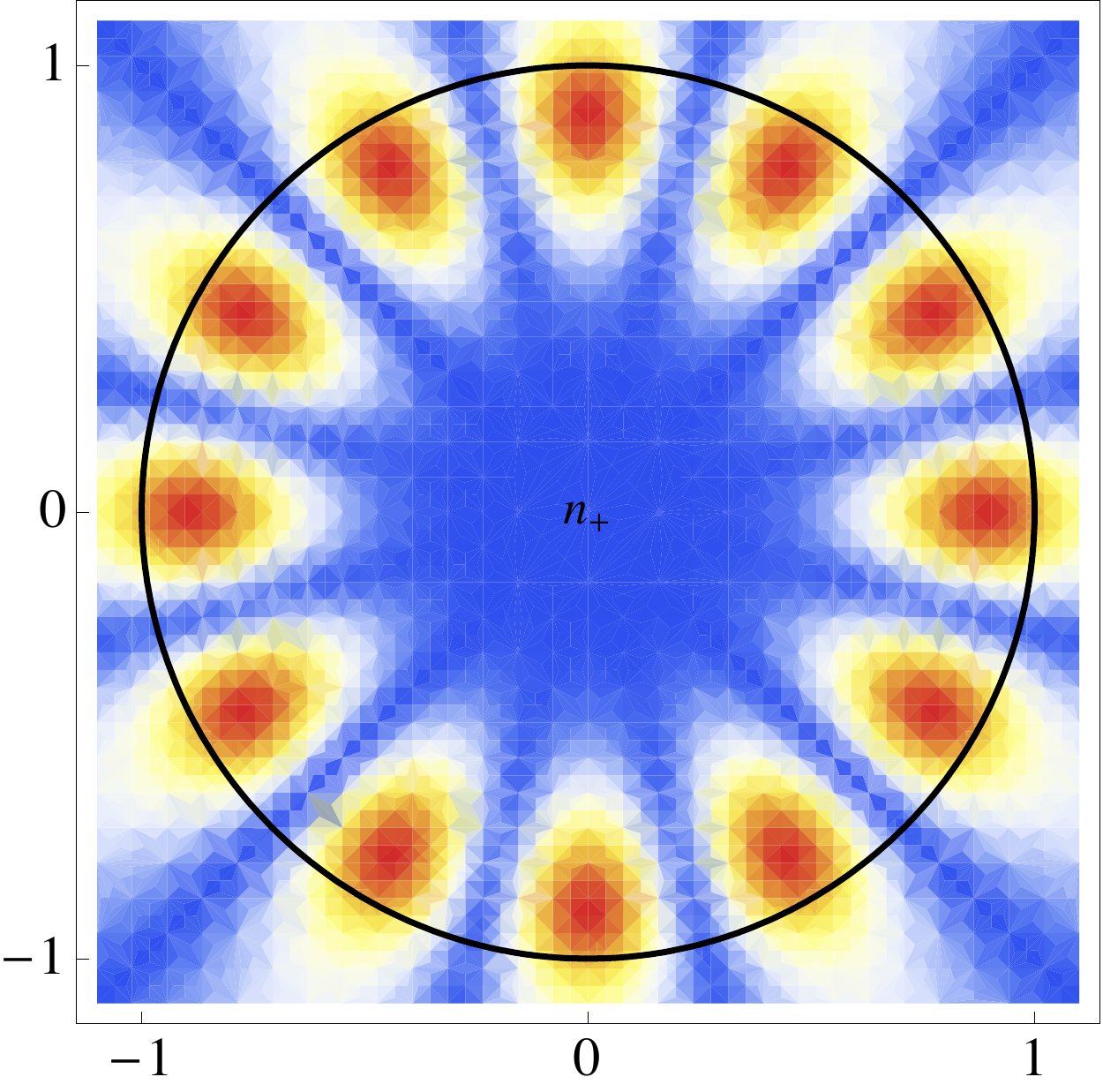}
 \includegraphics[width=2.5in]{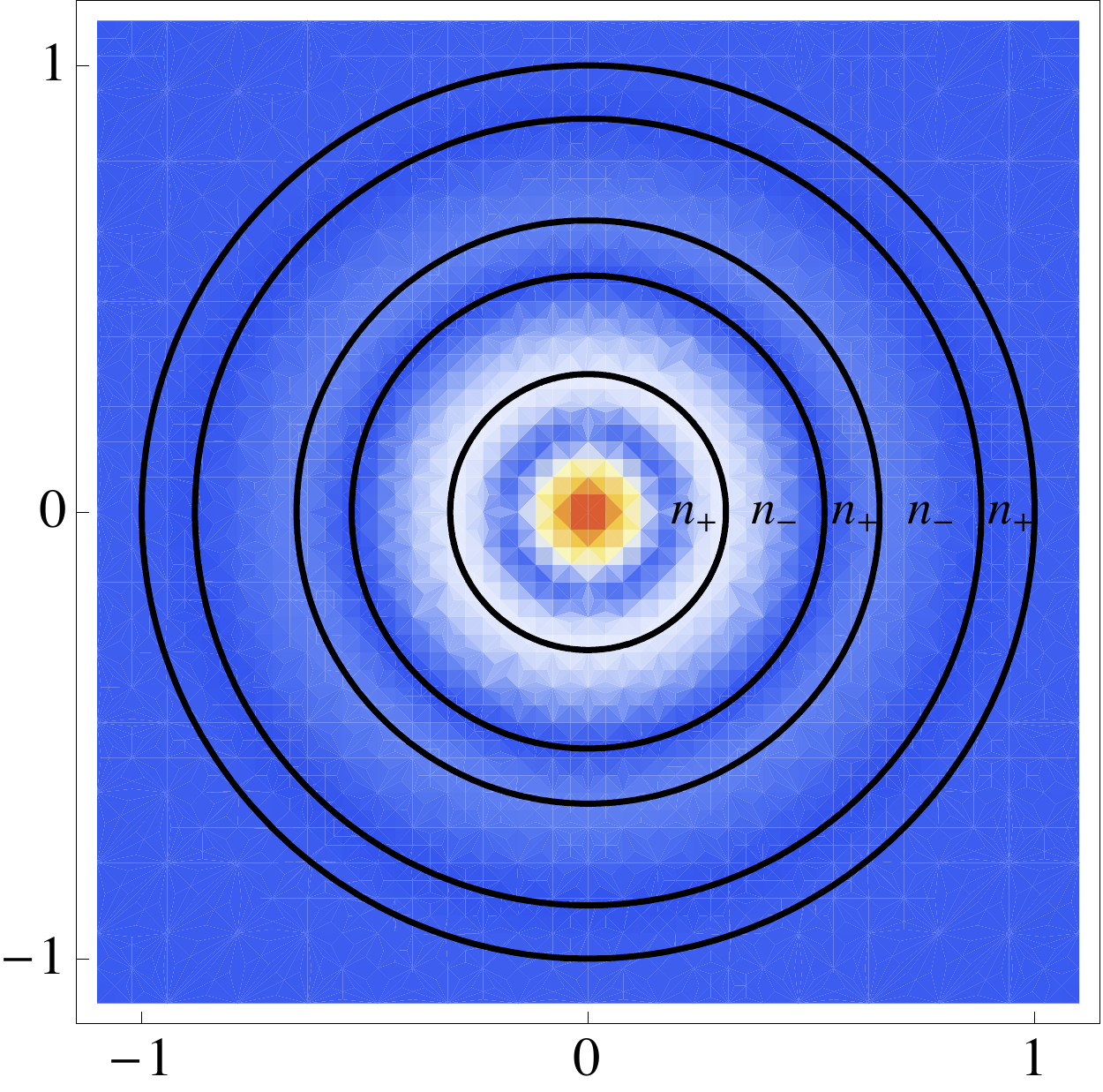}
 \caption{(left) The modulus of a mode with long lifetime ($\omega = 4.2 - .033 \imath$) due to (A) total internal reflection because of  large angular momentum ($\ell=6$) as in a spherical resonator. 
 (right)  Modulus of a mode with $\omega = 6.5 - .039 \imath$ confined by (B) interference effects, as in a Bragg resonator. }
 \label{fig:mech}
 \end{center}
 \end{figure}
\medskip

Figure \ref{fig:mech} illustrates the difference between mechanisms (A) and (B). 
   The left figure illustrates confinement via (nearly) total internal reflection;  the refractive index is a constant $n_+=2$
   inside the circular cavity, $\Omega = \{|x|<1\}$, and $n=1$ outside. Modes, $f_\ell(r)e^{\pm \imath \ell\theta}$,  $\ell=0,1,2,\ldots$  (in 2D) and $f_\ell(r)Y_\ell^m(\theta,\phi),\ \ |m|\le 2\ell +1,\ \ell=0,1,2,\ldots$ (in 3D) of increasing 
  ``angular momentum'', $\ell$,  have longer and longer lifetimes; the imaginary parts of the corresponding scattering resonances tend to zero.
The right figure displays radial confinement via interference effects; the refractive index profile consists of concentric annular regions  alternating between $n_-=1$ and $n_+=2$. 
   \medskip
   
\noindent{\bf Formulation of the Optimal Design Problem:}\ \  We seek
$n(\x)\in\calA$, a specified set of admissible  structures, having a resonance $\omega$ closest to the real axis ($\left|\Im\omega\right|\to\min$). 
Arbitrarily long confinement times may be achieved by either mechanism by allowing  for increasingly large $|\Re \omega|$. In mechanism (A), this corresponds to rays with ever shallower angle of incidence. In mechanism (B), there are more wavelengths for which the wave can destructively interfere with itself upon multiple reflections within the cavity. We therefore impose that $\left|\Re\omega\right|$ is no larger than a prescribed upper bound ($|\Re\omega|\le\rho < \infty$).

To obtain a precise formulation,  we first introduce admissible sets of structures. Let 
$\Omega\subset\mathbb{R}^d$ denote a fixed open and bounded set. Also, 
 let $0 < n_- < n_+< \infty$ be specified. Then,  our first  admissible set is given by:
 \begin{align}
\calA(\Omega,n_-,n_+) &:= \{ n \colon  {\rm supp}\ (n(\x)-1)\subset\overline{\Omega},\ \ \ 
 n_-\ \le\ n(\x)\ \le\ n_+ \} \label{A1-admiss}
\end{align}
In one spatial dimension, we shall take $\Omega=[0,L]$ and additionally consider the set of admissible structures which are symmetric:
\begin{align}
\calA_{sym}(L,n_-,n_+) &:= \{ n \in \calA ([0,L],n_-,n_+) \colon \ n(x)=n(L-x)  \}.\label{A3-admiss}
\end{align}
When the choice of $\Omega, n_-, n_+$ are unambiguous, we  simply write
 $\calA$ and $\calA_{sym}$.  
For $\rho>0$ define, for  $n(\x)\in\calA$, 
 $\Res_n^\rho\subset \Res_n$:
\begin{equation}
\Res_n^\rho := \{ \omega \in \Res_n \colon |\Re{\omega}| \leq \rho  \}.
\label{eqn:rhoRes}
\end{equation}
The minimal resonance width in the set $\Res_n^\rho$ is given by 
\begin{align} 
\Gamma^\rho[n] := \inf_{\omega \in \Res_n^\rho} |\Im \omega|. 
\label{eqn:obj}
\end{align}
If $\Res_n^\rho = \emptyset$, then we set $\Gamma^\rho[n] = \infty$.
  We  study the following \medskip
  
 \noindent {\bf Optimal Design Problem:}\  
\begin{equation}
\label{gen-min}
\Gamma_\star^\rho ({\cal A}) \ := \ \inf_{n \in \calA}  \ \Gamma^\rho[n]\ =\  
\inf_{n \in \calA}  \  \inf_{\omega \in \Res_n^\rho} |\Im \omega|.
\end{equation}
By the explicit example in section \ref{sec:simple-example}, $\Gamma_\star^\rho({\cal A})<\infty$, for the above choices of $\mathcal{A}$.

The corresponding lifetime of the optimal resonance mode is given by
$
\tau_\star({\cal A})  := \frac{1}{\Gamma_\star^{\rho}({\cal A})}.
$
An often used (dimensionless) quantity related to the lifetime is the \emph{quality factor} (Q-factor) of a scattering resonance, which is defined $Q := \frac{|\Re\omega|}{2|\Im \omega |}$. Several of the results of this paper have implications for the Q-factor since $Q < \frac{\rho}{2 \Gamma_\star^{\rho}({\cal A})}$.
 \subsection{Outline and summary of results}
 In section \ref{sec:scattering-resonances}, we derive variational-type identities for resonances (Proposition \ref{prop:epsilon}) in one spatial dimension. We use these identities  to show that for structures supported on a compact interval, $\Omega=[0,L]$, there is a general lower bound on the resonance width $|\Im\omega|$ of the following type:
 \begin{equation}
 \left| \Im\omega \right|\ \ge\ \alpha_\mathcal{A}\ e^{-\beta_\mathcal{A}\ |\Re\omega|^2}
 \label{form-of-bound}
 \end{equation}
 where $\alpha_{\cal A}$ and $\beta_{\cal A}>0$ depend on the  constraint set, $\mathcal{A}(\Omega,n_-,n_+)$; see Theorem~\ref{thm:aprioriBnd}. 
A consequence of \eqref{form-of-bound} is that for $L$ sufficiently large, the maximal lifetime resonance $\omega_\star$  satisfies $|\Re \omega_\star| > 0$; see Lemma \ref{cor:ReOmegaPositive}.    In Proposition \ref{prop:exclusionRect}, we also use a maximum principle argument to show that there exists a triangular resonance-free region in the lower 
half-plane. 

  In section \ref{sec:existence-of-optimizer}, we present results on the existence of a maximal lifetime resonance for dimensions $d=1, 2,\text{ and } 3$.
We prove, by studying the convergence of minimizing sequences, that there exists  a structure $n_\star(\x) \in \calA$ with a scattering resonance, $\omega_\star\in \Res_{n_\star}^\rho$, satisfying  \eqref{SRP}  with minimal width  $\left|\Im\omega_\star\right| = \Gamma_\star^\rho(\calA)$; see Theorem \ref{thm:exist}. 

 In  section \ref{sec:bangbang} we show, in dimensions $d=1, \ 2, \text{ and } 3$,   locally optimal structures $n_\star(\x) \in \calA$, are supported at the material bounds. That is, $n_\star(\x)$ is either $n_+$ or $n_-$ for almost every $\x\in\Omega$.

In sections \ref{sec:1dnum} and \ref{sec:1d-bragg}, we specialize to the case of one-dimensional structures where we can prove considerably more.  In section \ref{sec:1dnum}, we compute optimal structures using quasi-Newton optimization methods. These numerical observations motivate the investigations in Section \ref{sec:1d-bragg}, of the character of 1-d optimal refractive indices, $n_\star \in \calA_{sym}$. In particular, we prove that optimal structures for which the associated optimal resonance is simple, are step functions, \ie, have a finite number of transitions between $n_{+}$ and $n_{-}$; see Theorem \ref{prop:stepFun}. The key to the proof is Proposition \ref{prop:argu}, a monotonicity formula for the phase of a scattering resonance mode. We further prove a lower bound on the interval lengths on which $n=n_{+}$ and $n=n_{-}$; see Prop. \ref{prop:intLengthBnd}. 

In section \ref{sec:1d-bragg}, we show that the optimal structures are related to the well-known class of  Bragg structures, where $n(x)$ is constant on intervals whose length is one-quarter of the effective wavelength. In particular, we conjecture that the lengths of the intervals, $\delta_{\pm}$, for which $n(x) = n_{\pm}$ satisfies
$$
\delta_{\pm}(n_{+}, n_{-}, L) \rightarrow \frac{1}{4} \frac{2 \pi}{ |\Re \omega_{\star}| n_{\pm}} \equiv \frac{1}{4} \lambda_{\text{eff},\pm} 
\qquad \text{as $n_{+}\rightarrow \infty$ or $L \rightarrow \infty$};
$$ 
see Conjecture \ref{conj:Bragg}. $\lambda_{\text{eff},\pm}$ is called the effective wavelength. In App. \ref{sec:Bragg}, we discuss properties of infinite Bragg structures, and computationally demonstrate that they optimize the spectral gap to midgap ratio. 

\subsection{Related work}  \label{sec:relwork}

Results on the existence of optimal scattering resonances and general bounds on the imaginary parts of scattering  resonances for Schr\"odinger operators can be found in \cite{Harrell:1982uq,II:1986uq,Svirsky:1987fk}. Very recently, optimal designs have been considered in \cite{Karabash:2011uq}.  Our results for the Helmholtz equation make use of some of the arguments introduced in these papers.

There has been extensive investigation during the past several decades of ideal designs 
 for electromagnetic and photonic cavities, with a view toward 
device applications. 
Such design problems are typically formulated as an optimization problem for a particular figure of merit and solved using numerical optimization methods. 

The problem of maximizing the lifetime of a state trapped within a leaky cavity can be framed in several ways. 
The figure of merit can be taken to be the minimization of energy flux through the boundary \cite{Lipton:2003rq} or a measure of mode localization \cite{dobson-santosa,ABGKKLN-JP-2005}. In \cite{KaoSantosa}, the problem of minimizing $|\Im \omega|$ for a chosen resonance was investigated computationally in both one- and two-dimensions.  The 1-d problem considered here was also studied computationally in \cite{OptimizationOfScatteringResonances}. 
This work focuses on the optimization of $\sigma$  to minimize $|\Im \omega|$ which satisfies outgoing solutions of the equation
$
\partial_x \sigma \partial_x u(x,\omega) + \omega^2 n^{2} u(x,\omega) = 0.
$
In particular, the variations $\frac{\delta \omega}{\delta n}$ and $\frac{\delta \omega}{\delta \sigma}$ are formally computed. 
In \cite{Scheuer:2006fk}, transfer matrix methods were used to design low-loss 2D resonators with radial symmetry.  
In each of these papers, gradient-based optimization methods were used to solve the optimization problem.

Genetic algorithms have also been employed  to minimize energy flux through the boundary \cite{A.-Gondarenko:2006fk,Gondarenko:2008uq}. In \cite{Englund:05,Geremia2006,Felici2010} the ``inverse method''  is employed, where a desired mode shape is chosen and then the material properties which produce that mode are found algebraically. 
  In \cite{bauer2008}, the  time-dependent problem is solved to steady state using a finite-difference method with perfectly matched layers to approximate the outgoing boundary conditions. The design problem is solved using a  Nelder-Mead method. 

An important, related class of problems is to find photonic structures with large spectral band gaps. For the one-dimensional case, see the further discussion in Appendix \ref{sec:Bragg}. Structures with optimally large band gaps have been proven to exist \cite{Cox:1999qf,Osting2012} and numerical methods have been applied to finding them \cite{Cox:2000qc,Burger:2004fk,Kao:2005pt}. In \cite{Sigmund:2003jo}, topology optimization was used to find photonic crystals with optimally large bandgaps and also which optimally damp or guide waves. In \cite{Sigmund:2008wb}, properties of photonic crystals with optimally large bandgaps are investigated. 

One property of optimal structures for  \eqref{gen-min} is that they are piecewise constant structures which achieve the material bounds, \ie, they are {\it bang-bang} controls. 
This property is also realized in a number of optimization problems for eigenvalues of self-adjoint operators \cite{Krein:1955ye,Cox1990,Cox19902,Chanillo:2000,Osting2012} as well as for Schr\"odinger resonances \cite{II:1986uq}. 
In \cite{potDesign,OstingThesis} the authors consider the problem of maximizing the lifetime of a state coupled to radiation by an {\it ionizing} perturbation. For this class of problems, optimizers are interior points of the constraint set. 

\section*{Acknowledgements} 
 Some of the investigations of this article were motivated by discussions with P. Heider \cite{Heider-Weinstein-unpubl}.
 The authors wish to thank A. Barnett, D. Bindel, P. Heider, S.~G. Johnson, R.~V. Kohn,  G. Ponce,  J.~V. Ralston, F. Santosa, O. Savin, D. Tataru, and C.~W. Wong  for stimulating interactions. Finally, we thank the referees for their helpful comments.
B. Osting was supported in part by U.S. NSF Grant No. DMS-06-02235, EMSW21- RTG: Numerical Mathematics for Scientific Computing and NSF Postdoctoral Fellowship DMS-11-03959. M.~I. Weinstein was supported in part by  U.S. NSF Grants DMS-07-07850 and  DMS-10-08855.

\section{Scattering resonances: variational identities, and \\ bounds in one dimension}
\label{sec:scattering-resonances}
In this section we  derive  variational-type identities and use them to obtain universal inequalities for one-dimensional scattering resonances in terms of the support of $n(x)-1$ and  pointwise bounds on $n(x)$. 

\subsection{Variational identities}\label{sec:variational-identities} 
In dimension one, a scattering resonance $u(x,\omega)$ satisfying \eqref{SRP} on the domain $\Omega \equiv [0,L]$  is a weak solution of
\begin{subequations}
\label{1dSRP}
\begin{align}
\label{1dSRPa}
&\partial_x^2 u(x;\omega) + \omega^2 n^2(x) u(x;\omega)\ =\ 0,\ \ 0\le x\le L\\
\label{1dSRPb}
& \partial_x u(0;\omega) = -i\omega u(0;\omega),\ \  \partial_x u(L;\omega) = i\omega u(L;\omega). 
\end{align}
\end{subequations}
The following proposition gives a variational-type identity for the scattering resonance problem in dimension one.

\begin{prop} 
\label{prop:epsilon}
Let $n\in\calA$, as defined in \eqref{A1-admiss} be a refractive index and let $\left(\omega,u(\x,\omega)\right)$ be a one dimensional scattering resonance pair on the domain $\Omega \equiv [0,L]$, {\it i.e.} a weak solution of \eqref{1dSRP}. Then
\begin{subequations}
\label{eqn:epsilon}
\begin{align}
\Re(\omega^2) &=  \frac{ \int_0^L  \left|u'(\cdot,\omega)\right|^2  + \Im \omega \left( \left|u(0,\omega)\right|^2 + \left|u(L,\omega)\right|^2 \right)}{ \int_0^L n^2 \left|u(\cdot,\omega)\right|^2 } \\
\Im(\omega^2) &= - \frac{\Re \omega \left( \left|u(0,\omega)\right|^2 + \left|u(L,\omega\right)|^2 \right)}{ \int_0^L n^2 \left|u(\cdot,\omega)\right|^2 }. \label{Imomega2}
\end{align}
\end{subequations}
Furthermore,  \eqref{Imomega2} and $\Im\omega<0$ imply:
\begin{equation}
\label{eqn:tau}
|\Im \omega|  =  \frac{  |u(0,\omega)|^2 + |u(L,\omega)|^2 } {2  \int_0^L n^2 |u(\cdot,\omega)|^2 } .
\end{equation}
\end{prop}

\begin{proof}
Multiply  \eqref{1dSRPa} by $\overline{u(\x,\omega)}$ and integrate over $\Omega$ to obtain
\begin{align*}
\left(\Re(\omega^2) +  \imath \Im(\omega^2) \right) \int_\Omega n^2 |u|^2 \ud \x 
= \int_\Omega | u_{x}|^2 \ud \x - \imath \omega \left( | u(0,\omega) |^{2} + |u(L,\omega)|^{2} \right).
\end{align*}
Identifying real and imaginary parts yields 
Eq. \eqref{eqn:epsilon}. 
Equation \eqref{eqn:tau} follows from Eq. (\ref{eqn:epsilon}b) 
 and the relationship $| \Im(\omega^2)| = 2 |\Re \omega| |\Im \omega|$. 
\end{proof}

\subsection{Lower bounds for resonances of the one-dimensional Helmholtz equation}
\label{sec:1d-res-bound}
We use the variational-type identities  from Sec.  \ref{sec:variational-identities} to show the following universal inequality for one-dimensional scattering resonances. 
\medskip

\begin{thm} 
\label{thm:aprioriBnd}
Let $\Omega = [0,L] \subset \mathbb R^1$ and $n  \in \calA(\Omega,n_-,n_+)$. 
For any scattering resonance $\omega \in \Res_n$  and  $\xi>0$, 
\begin{align*}
|\Im \omega| \geq  \min \left[\,  \xi, \,\, 
\frac{ 3 \exp \left(-  \left( |\Re \omega|^2 + \xi^2 \right) n_+^2 L^2\right) }
{n_+^2L \left(3 + L^2 \left( |\Re \omega|^2 + \xi^2 \right) \right)}
\,  \right].
\end{align*}
\begin{equation}
\label{eq:aprioriBndw}
\textrm{In particular}\ n_+ > e^{-1}\ \implies\ |\Im \omega| \geq \frac{3 \exp \left( - n_+^2 L^2 |\Re \omega|^2 \right)}
{e L \left(1 + 3 n_+^2 + n_+^2 L^2 |\Re \omega|^2 \right) }. 
\end{equation}
\end{thm}

\begin{rem} 
\label{cor:ReOmegaPositive}
In one-dimension, if $\rho$ is sufficiently large,  the optimal resonance, $\omega_\star$, satisfies  $|\Re \omega_\star| > 0$. This observation follows from the construction of explicit examples where $|\Im \omega| < \frac{3}{eL( 1 + 3n_{+}^{2})}$; this includes simple piecewise constant structures  \cite{Heider-Weinstein-unpubl} or  numerically constructed refractive indices (see Section \ref{sec:1dnum}).
\end{rem}

\medskip

We follow the strategy of \cite{Harrell:1982uq} to prove Theorem \ref{thm:aprioriBnd}. Theorem  \ref{thm:aprioriBnd} relies on the following two lemmata, which we shall prove first. 
In this section only, we normalize the resonance state by assuming, without loss of generality,  
\begin{equation}
\label{eqn:normAssum}
1 = u(0) \leq |u(L)|^2
\end{equation}
(otherwise we make the substitution $x \mapsto L-x$).  

\begin{lem} 
\label{lem:pntwisebndu}
Let $u(x;\omega), \omega$ denote a scattering resonance pair for the one-dimensional scattering resonance problem, \eqref{1dSRP}, defined for  $0\leq x \leq L$. Then,  we have the pointwise bound
\begin{align} 
\label{eqn:pntwisebndu}
|u(x;\omega) | \leq \sqrt{1 + |\omega|^2 x^2} \exp \left( |\omega|^2 \int_0^x (x-y) n^2(y) \ud y \right) . 
\end{align}
\end{lem}

\begin{proof}
Assuming \eqref{eqn:normAssum}, the boundary condition \eqref{1dSRPb} is written $u(0) = 1,\  u_x(0) = -i \omega u(0)$ and $ u_x(L) = i \omega u(L)$.\  Integrating twice, we obtain the integral equation
$$
u(x) = 1 - i \omega x - \omega^2 \int_0^x  \int_0^y n^2(z) u(z) \ud z \ud y. 
$$
Integrating the outer integral by parts, we obtain 
$
\left| \int_0^x \int_0^y n^2(z) u(z) \ud z \ud y \right| \leq  \int_0^x (x-y) n^2(y) |u(y)| \ud y
$
and thus 
$
|u(x)| \leq \sqrt{1 + |\omega|^2 x^2 } + |\omega|^2 \int_0^x (x-y) n^2(y) |u(y)| \ud y. 
$
Equation (\ref{eqn:pntwisebndu}) now follows from Gronwall's inequality. 
\end{proof}

\begin{lem}
\label{lem:intn2u2} Assuming the same hypotheses as in  Lemma \ref{lem:pntwisebndu}, 
\begin{align}
\label{eqn:intn2u2}
\int_0^L n^2 |u|^2 \ud x \leq 
n_+^2 L \exp \left(|\omega|^2 n_+^2 L^2\right) \left(1 + |\omega|^2 L^2/3\right)
\end{align}
where $n_+ = \max_{x\in(0,L)} n(x)$. 
\end{lem}
\begin{proof}
Using Lemma \ref{lem:pntwisebndu}, we compute
\begin{align*}
\int_0^L n^2 |u|^2 \ud x &\leq \int_0^L n^2(x) (1+  |\omega|^2 x^2) \exp\left(2 |\omega|^2  \int_0^x (x-y)n^2(y) \ud y\right) \ud x\\
&\leq n_+^2 \exp \left(2 |\omega|^2 \int_0^L (L-y)n^2(y) \ud y \right)  \int_0^L 1+ |\omega|^2 x^2  \ud x\\
&\leq n_+^2 \exp \left(|\omega|^2 n_+^2 L^2\right) \left(L + |\omega|^2 L^3/3\right)
\end{align*}
as desired.
\end{proof}

\begin{proof}[Proof of Theorem \ref{thm:aprioriBnd}.]
Using (\ref{eqn:tau}), (\ref{eqn:normAssum}), and Lemma \ref{lem:intn2u2}, we compute
\begin{subequations}
\begin{align*}
|\Im \omega|
&=   \frac{ |u(0)|^2 + |u(L)|^2 }{2 \int_0^L n^2 |u|^2 \ud x }\ 
\geq \frac{ \exp \left( - |\omega|^2  n_+^2 L^2\right)}{n_+^2L \left(1 + |\omega|^2  L^2/3\right)}\\
&=  \frac{ \exp \left(- \left( |\Re \omega|^2 + |\Im \omega|^{2}\right)  n_+^2 L^2\right)}{n_+^2 L \left(1 + \left( |\Re \omega|^2 +  |\Im \omega|^{2}\right)  L^2/3\right)}\ 
 \equiv f(|\Im \omega|).
\end{align*}
\end{subequations}
This is a nonlinear inequality for $|\Im \omega|$. Note that  $f(x)$ is a monotonically decreasing function for $x\geq0$ with $f\downarrow 0$ as $x\uparrow \infty$. 
Thus, for $\xi \geq |\Im \omega|$, $f(\xi) \leq f( |\Im \omega|) \leq |\Im \omega |$. Thus for all $\xi > 0$, 
$|\Im \omega| \geq \min [ \xi, f( \, \xi \, ) ]$.
To obtain the optimal bound, one would choose $\xi = \xi_0$ such that $\xi_0 = f( \, \xi_0 \, )$.  
For simplicity, we choose $\xi_0 =  (n_+ L)^{-1}$. 
If $n_+ > e^{-1}$, we find that $\min[\xi_0, f(\xi_0)] = f(\xi_0)$ for all $\Re \omega$ and \eqref{eq:aprioriBndw} follows. 
\end{proof}

\begin{rem}
Theorem \ref{thm:aprioriBnd} also gives an upper bound for the quality factor, defined $Q := \frac{|\Re \omega| }{2 |\Im \omega|}$. In particular this bound shows that $Q\downarrow 0$ as $\Re \omega \downarrow 0$. 
\end{rem}

The following proposition shows that there is a triangular resonance-free region in the lower-half complex plane. 

\begin{prop}
\label{prop:exclusionRect}
Let $n \in \calA_{sym}$, \ie, $n_-<n(x)<n_+$ and $n(L-x)=L(x)$. If $\omega\in \text{Res}_n$ is a one-dimensional Helmholtz resonance satisfying \eqref{1dSRP} with $d=1$, then
$\omega \notin \{\omega\colon |\Im\omega| > |\Re \omega| \text{ and } |\Im \omega| \leq \frac{1}{n_+^2 L} \}$.
\end{prop}

\begin{proof} 
The proof of this theorem follows \cite{II:1986uq}. Let $|\Im \omega| > |\Re \omega|$ and we'll show that $|\Im \omega | > \frac{1}{n_+^2 L}$. Using Eq. \eqref{eqn:tau} and $|u(0)| = |u(L)|$ (see Prop. \ref{prop:uSym}), we have
\begin{equation}
\label{eqn:ratioTri}
|\Im \omega|  \geq \frac{|u(0)|^2} {n_+^2 \int_0^L |u|^2 \ud x }  . 
\end{equation}
Kato's inequality \cite{RS2} and  $ |\Im \omega| \geq |\Re \omega| \ \Rightarrow \ \Re(\omega^2) \leq 0$ then give 
$\Delta |u| \geq \Re \left( \frac{\overline{u}}{|u|} \Delta u\right) = -  \Re(\omega^2) n^2 |u| \geq 0$.
We now apply the maximum principle to the subharmonic function $|u(x)|$ to obtain
$ |u(x)| \leq |u(0)| $.
It  now follows from Eq. \ref{eqn:ratioTri} that $|\Im \omega | > \frac{1}{n_+^2 L}$. 
\end{proof}

\section{Existence of a solution for the spectral optimization problem}
\label{sec:existence-of-optimizer}

In this section, we consider the spectral optimization problem in dimension $d=1,2,3$ with admissible set $\calA (\Omega,n_-,n_+)$, as defined in Eq. \eqref{A1-admiss}, the set of  $n(\x)$ satisfying upper and lower bounds on the compact set $\overline{\Omega}$ with $n(\x)\equiv1$ for $\x\notin\overline{\Omega}$, {\it i.e.}
\begin{equation}
n_-\le n(x)\le n_+,\ \ x\in\Omega,\ \ {\rm and}\ \ \ n(x)\equiv1,\ \ \ x\notin\Omega. 
\nonumber\end{equation}
Recall ${\rm Res}_n^\rho$ as defined in \eqref{eqn:rhoRes} is the set of scattering frequencies, $\omega$, for the structure $n(\x)$ such that $|\Re \omega| \leq \rho$. 
\medskip

\begin{thm}
\label{thm:exist}  Consider the scattering resonance problem on $\mathbb{R}^d,\ d=1,2,3.$
Fix $\rho\ge0$.  
Assume that there exists $n\in \calA$ such that $\Res_n^\rho \neq \emptyset$. Then the double infimum, defined in  \eqref{gen-min}, is strictly  positive and is attained for an admissible structure. That is, there exists $n_\star \in \calA$, with associated longest-lived resonance mode, $u_\star$,  of frequency $\omega_\star \in \Res_n^\rho$ and such that $\left|\Im\omega_\star \right| = \Gamma^\rho_\star(\calA) >0$.  
\end{thm}
\begin{proof}
 Since $\Res_n^\rho\ne\emptyset$ and $\Res_n^\rho\subset\{ \omega \colon \ \Im\omega \leq 0\ \}$, we have $0\le \left|\Im\omega\right|<\infty$ and  there is a minimizing sequence
$\left\{ n_m \right\}_{m=1}^\infty \subset {\cal A}$, such that 
\begin{equation}
\inf\{\ \left|\Im\omega\right| \colon \omega\in \Res_{n_m}^\rho \ \} \downarrow \Gamma^\rho_\star\ge0, \  \ \text{as}\ m\uparrow\infty\ .
\end{equation}
We first show that $\Gamma_\star^\rho$ is attained and then conclude the proof by showing $\Gamma_\star^\rho>0$.

Let $\left(\ \omega_m,u_m(\x,\omega_m)\ \right)$ denote a sequence of resonance pairs corresponding to this minimizing sequence of structures in $\calA$.  Since $[-\rho,\rho]$ is compact,  there exists a  convergent subsequence, which we continue to denote by $\{\omega_m\}$,  with $\omega_m\to\omega_\star$.

Since $\Omega$ is bounded, we have  $\Omega\subset B_R(0)$, where $B_R(0)$ denotes the open ball of radius $R$ about the origin.
By linearity and boundedness of $B_R(0)$ we can impose the normalization $\|u_m\|_{L^2(B_R(0))} = 1$. 
Squaring the differential equation for $u_m$, 
$-\Delta u_m(x)=\omega_m^2 n_m^2(\x) u_m(\x)$, and integrating over $B_R(0)$, we obtain
\[ \int_{B_R(0)} \left|\Delta u_m(\x)\right|^2\ d\x\ \le \ |\omega_m|^4\ n_+^4\ \int_{B_R(0)} |u_m(\x)|^2 d\x\ . \]
Therefore, since $\|u_m\|_{L^2(B_R(0))}=1$, 
 the sequence $\{u_m\}$ is uniformly bounded in $H^2(\mathcal O)$, for any open $\mathcal O$, whose closure is a compact subset of $B_R(0)$. Thus, by Rellich's Lemma, for any $s<2$ there exists $u_\star\in H^s(\Omega)$ and  a strongly convergent subsequence converging to $u_\star$. Moreover, $ u_m$ is uniformly H\"older continuous with exponent $\alpha\in(0,1/2)$. Thus, $\left\{u_m(\x)\right\}$ is uniformly bounded and equicontinuous. Therefore, there exists a subsequence, again denoted $\{u_m\}$, such that 
 $u_m\to u_\star$ uniformly in $\overline{\Omega}$. It follows that $\|u_\star\|_{L^2(\Omega)}=1$ and thus $u_\star$ is nonzero. 
 
The uniform  bounds $n(\x) \subset [n_-, n_+]$ imply that the sequence $\{n_m\}$ is uniformly bounded in $L^2(\Omega)$, and therefore  along some subsequence converges weakly in $L^2(\Omega)$ to some $n_\star$, \ie, $n_m \warrow n_\star$ with $n_\star \in \calA$. Furthermore, $n_m(\x) \rightarrow n_\star(\x)$ a.e. in $\Omega$.

It remains to show that $(u_\star(\cdot,\omega_\star),\omega_\star)$ is a resonance pair, \ie,
 $u_\star$ is non-trivial (established just above) solution of \eqref{SRP}.
For each $m\in \mathbb N$, we also have that  
\begin{equation}
u_m(\x) =  \omega_m^2 \int_\Omega G(|\x-\y|,\omega_m) [n_m^2(\y) - 1]  u_m(\y) \ud \y. 
\label{um-eqn}
\end{equation}
As shown above, the right hand side of \eqref{um-eqn} converges to $u_\star(\x)$ uniformly on $\overline{\Omega}$. Therefore to establish \eqref{SRP}, it suffices to show that
\begin{equation}
\omega_m^2 \int_\Omega G(|\x-\y|,\omega_m) [n_m^2(\y) - 1]  u_m(\y) \ud \y \ \longrightarrow \ \omega_\star^2 \int_\Omega 
G(|\x-\y|,\omega_\star) [n_\star^2(\y) - 1]  u_\star(\y) \ud \y 
\label{um2ustar}
\end{equation}
for  $\x\in\overline{\Omega}$. 
For each $\x\in\overline{\Omega}$, the integrand of \eqref{um2ustar} converges in $\y$ pointwise a.e. to the expression with the subscript $m$ replaced by $\star$. 
Also for each $\x\in\overline{\Omega}$, the integrand (as a function of $\y$) is dominated by an integrable function; there exists a constant $K$ such that 
\[ \left|G(|\x-\y|,\omega_m) [n_m^2(\y) - 1]  u_m(\y)\right|\ \le K\ |\x-\y|^{2-d},\] 
uniformly in $m$. 
Hence, by the dominated convergence theorem,   \eqref{um2ustar} holds.

   Finally, we claim that $\Gamma_\star^\rho>0$. Suppose not. Then $\Gamma_\star^\rho=0$ and  the scattering resonance problem \eqref{SRP} has a non-trivial solution with real frequency $\omega_\star$. By the unique continuation principle \cite{Colton:1998fk,Koch-Tataru:06} $u_\star\equiv0$, a contradiction.
\end{proof}
\medskip

\begin{rem} 
In the example of Appendix \ref{sec:simple-example}, we showed that in dimensions $d=2,3$ for  $\Omega = \{ |\x|<a\}$ and $n(\x) = 1+ n_0 \mathbf{1}_\Omega$, the resonances approach the real axis as the angular momentum $\ell \uparrow \infty$.  In section \ref{sec:1dnum}, we show that in dimension $d=1$, the sequence $n_\star^k$ for increasing $\rho^k \uparrow \infty$ is such that $| \Im \omega_\star^k | \downarrow 0$ with $ | \Re \omega_\star^k |\approx \rho^k \uparrow \infty$. Thus for $d=1,2,3$ an optimal solution for the limit $\rho \uparrow \infty$  is not achieved.
This result contrasts the behavior of optimal resonances of the Schr\"odinger operator  \cite{Harrell:1982uq,II:1986uq,Svirsky:1987fk}.
\end{rem}

\section{Local optimizers are piecewise constant structures which saturate the constraints}
\label{sec:bangbang}
In this section, we focus on properties of locally  optimal solutions of the design problem
\begin{equation}
\label{eqn:designProb}
\min_{n \in \calA}  \ \Gamma^\rho[n] .
\end{equation}
We begin by computing the variation of $\omega$ with respect to changes in the index $n(\x)\in 
\mathcal{A}\subset L^\infty$. Assume a mapping $\mathcal J:\mathcal{A}\to\mathbb{C},\ n\mapsto\mathcal{J}[n]$.
We say the mapping $\mathcal{J}$ is Fr\'echet  differentiable at  $n\in\mathcal{A}$ if there exists a mapping $n\in\mathcal{A}\mapsto \frac{\delta \mathcal J}{\delta n}\in L^1(\Omega)$ such that for all 
$\delta n\in L^\infty$ such that $n+\delta n\in\mathcal{A}$ and $\|\delta n\|_{L^\infty}$ sufficiently small we have
\[ \left|\ \mathcal J[n + \delta n] - \mathcal J[n] -\left\langle \frac{\delta \mathcal J}{\delta n}, \delta n  \right\rangle\ \right|\ \to\ 0,\ {\rm as}\ \| \delta n \|_{L^\infty(\Omega)}\to\ 0\ . \]
Here $\langle f,g \rangle  = \int_{\Omega} \overline{f(\x)}\ g(\x) \ud \x$, defined for $f\in L^1$ and $g\in L^\infty$.
\medskip

\begin{prop}
\label{prop:gradomega}
Let $(\omega, u(\x,\omega))$ be a nondegenerate scattering resonance pair of the scattering resonance problem \eqref{SRP} for index of refraction,  $n(\x)$.  Then 
\begin{enumerate}
\item The first variation of $\omega[n]$ with respect to $n(\x)$ is given by 
\begin{equation}
\label{eqn:gradomega}
\frac{\delta \omega}{\delta n}(\x) =  - 2   \overline{\alpha}  \ \overline{\omega}^2\ n(\x)\  \overline{u(\x)}^2 
\end{equation}
where $\alpha \in \mathbb C \setminus \{0\}$ depends on $u$. 
\item In one-dimension, with $\Omega = [0,L]$, the first variation is given by 
 \eqref{eqn:gradomega} with 
 \begin{align}
 \label{eqn:1dalpha}
 \alpha^{-1} = 2\omega \int_0^L n^2 u^2 + \imath [u^2(0) + u^2(L)] 
 = \frac{1}{\omega} \int_{0}^{L} u_x^2 + \omega^2 n^2 u^2
 \end{align} 
\end{enumerate}
\end{prop}
\begin{proof} The proof of Proposition \ref{prop:gradomega} is given in Appendix \ref{sec:variations}. \end{proof}

\begin{rem}The following are properties of the variation $\frac{\delta \omega}{\delta n}(\x)$:
\begin{enumerate}
\item   From the explicit expression for $\frac{\delta \omega}{\delta n}(\x)$, we see that it  is invariant 
under  $u(\x)\mapsto cu(\x),\ \ c\in\mathbb{R}$.
\item For finite degeneracy, \eqref{eqn:gradomega} remains valid, except the correct (perturbation dependent) scattering resonance mode must be chosen from the eigenspace. 
\item Suppose $d=1$ and $\Omega=[0,L]$. If $n(x)$ is symmetric, \ie, $n(x)=n(L-x)$, then  $\frac{\delta \omega}{\delta n}(x)$ is also symmetric. This follows from Prop. \ref{prop:uSym}. 
\item If $d=2$ or $d=3$, and $n(\x)$ is radially symmetric, then $\frac{\delta \omega}{\delta n}(\x)$ is \emph{not} radially symmetric if the mode $u(\x,\omega)$ has a nontrivial angular dependence. 
\end{enumerate}
\end{rem}
\medskip

\begin{prop}
\label{prop:bangbang}
Let $n_\star(\x) \in \calA$ be a  local minimizer of the design problem \eqref{eqn:designProb}. If $\omega_{\star}$ is nondegenerate and $| \Re \omega_\star | < \rho $,  then $n_\star(\x)$  is piecewise constant and attains the material bounds, \ie,
$\left[ n_\star(\x) - n_- \right] \left[ n_\star(\x) - n_+ \right] = 0  \text{ for almost every  } \x \in \Omega$. 
\end{prop}
\medskip

\begin{rem}
Computational evidence for Proposition \ref{prop:bangbang} has been reported on in 
\cite{KaoSantosa,OptimizationOfScatteringResonances}. 
A similar result in one-dimension with slightly different boundary conditions is given in \cite{Karabash:2011uq}. 
This phenomena has also been studied in self-adjoint systems \cite{Krein:1955ye,Cox1990,Cox19902} and for the analogous problem for Schr\"odinger resonances \cite{II:1986uq}. In control theory, $n(\x)$ would be referred to as a  ``bang-bang control.''
\end{rem}
\begin{rem}
 Proposition \ref{prop:bangbang} has significance in applications because
\begin{enumerate}
\item the optimization problem is reduced to finding the $(d-1)$-dimensional interface between regions with constant $n(\x)$, and 
\item the computation of Helmholtz resonances for piecewise constant $n(\x)$ can be performed more efficiently using methods which utilize the free-space Green's function \eqref{eqn:explicitG-1}, {\it e.g.},  meshless methods, the method of particular solutions, and layer potential methods,
\item standard manufacturing techniques for photonic crystals involve creating air holes in a dielectric media \cite{pc-book}. Thus, in practice, $n(\x)$ only takes two values.
\end{enumerate}
\end{rem} 

\bigskip

\begin{proof}
Let $A\subset \Omega$ be defined to be the set
$A: = \{ \x \in \Omega \colon  n_- < n_\star(\x) < n_+ \}$ and 
$B\subset A$ be an arbitrary set. Local optimality requires that 
\begin{align*}
\left\langle \frac{\delta \Im \omega_\star}{ \delta n} , 1_B \right\rangle = 0 
&\iff \Im \left( \alpha \omega_\star^2 \int_B u_\star^2 \right) = 0\ \iff \Im \left( \alpha \omega_\star^2 u_\star^2(\x) \right) = 0  \quad \forall \x \in A
\end{align*}
since $B$ is arbitrary. We conclude that $\alpha \omega_\star^2 u_\star^2(\x)$ is a real-valued function on $A$. 

In the neighborhood of any $\x$ such that $u_\star(\x)\neq 0$, we can choose the sign such that $v(\x):= \sqrt{ \pm \alpha \omega_\star^2 u_\star^2(\x)}$ is a real-valued function which satisfies the complex  Eq. \eqref{helm}:
$
\Delta v + \Re(\omega_\star^2) n^2 v = - \imath \Im(\omega_\star^2) n^2 v.
$
Since the left-hand side of this equation is purely real and the right-hand side is purely imaginary, we conclude that $v\equiv 0$. Thus $u_\star \equiv0$ on $A$. 

Since $u(\x)$ satisfies \eqref{helm}, we know that by the unique continuation principle \cite{Colton:1998fk,Koch2001}, it cannot vanish on an open set of $\Omega$, or else $u_\star(\x)\equiv 0$, a contradiction. 
Thus $\text{meas}(A)=0$ and we identify it with the zero level set of $\Im \left( \alpha \omega_\star^2 u_\star^2(\x) \right)$. 
\end{proof}

\medskip

Proposition \ref{prop:bangbang} is further strengthened by the following 

\medskip

\begin{prop} 
\label{corr:nChar}
Let $n_\star(\x) \in \calA$ be a local minimizer of the design problem \eqref{eqn:designProb} 
and  $(u_\star,\omega_\star)$ denote the corresponding minimizing scattering resonance pair. 
If $\omega_{\star}$ is nondegenerate and $| \Re \omega_\star | < \rho $, then 
\begin{equation}
\label{eqn:nChar}
n_\star(\x) = \begin{cases}
n_+ &\quad  \Im \left( \alpha \omega_\star^2 u_\star^2(\x) \right) > 0 \\
n_- &\quad    \Im \left( \alpha \omega_\star^2 u_\star^2(\x) \right) < 0
\end{cases}
\end{equation}
where  $\alpha\in \mathbb C$ is defined as in Prop. \ref{prop:gradomega}. 
\end{prop}
\begin{proof}
By Prop. \ref{prop:bangbang}, we may decompose $\Omega = \Omega_+ \cup \Omega_-$,  where  
\begin{align*}
\Omega_+ &= \{ \x \in \Omega \colon n_\star(\x) = n_+ \} \ \textrm{and}\ 
\Omega_- = \{ \x \in \Omega \colon n_\star(\x) = n_- \} . 
\end{align*}
Let $A\subset \Omega_+$ and $1_A$ be the indicator function for the set $A$. Local optimality of $n_\star$ implies that
\begin{align*}
\left\langle \frac{\delta \Im \omega_\star}{ \delta n} , 1_A \right\rangle < 0 
&\iff \Im \left( \alpha \omega_\star^2 \int_A u_\star^2 \right) > 0\
\iff \Im \left( \alpha \omega_\star^2 u_\star^2(\x) \right) > 0  \quad \forall \x \in \Omega_+
\end{align*}
since $A$ is an arbitrary set. The statement for $\Omega_-$ in \eqref{eqn:nChar} follows similarly. 
\end{proof}

\section{Computation of optimal one-dimensional $n(x)$}
\label{sec:1dnum}
In this section we describe one-dimensional computations that will be used to motivate analytical results in section \ref{sec:1d-bragg}. Computationally, we solve a local variant of the optimization problem (double-infimum) in \eqref{eqn:obj}. Following \cite{OptimizationOfScatteringResonances,KaoSantosa}, we pick a resonance $\tilde{\omega}$ and improve the structure $n(x)$, by local gradient methods,  to extend the lifetime of that particular resonance. 
Thus we solve the local optimization problem  
\begin{equation}
\label{eqn:optOneRes}
\textrm{Seek local minima of}\ \  |\Im \tilde{\omega}[n] |,\ \textrm{subject to}\ \ 
 n\in \calA.
\end{equation}
The local minimizers of Eq. \eqref{eqn:optOneRes} which additionally satisfy $|\Re {\tilde \omega}_\star | < \rho$ are local minima of Eq. \eqref{eqn:designProb}. 

\subsection{Computational method} \label{sec:CompMeth}
Below, we refer to the forward problem as the computation of the resonances for a given $n(x)$ and the optimization problem as the solution of Eq. \eqref{eqn:optOneRes}.

\paragraph{Forward problem} For a one-dimensional ($d=1$),  piecewise-constant, refractive index $n(x)$,  the resonances satisfying Eq. \eqref{1dSRP} are the roots of a nonlinear system of equations obtained by imposing transmission conditions at the material discontinuities.
In  \cite{OptimizationOfScatteringResonances}, this system of equations is derived and Newton's method is applied for finding the roots of this system.  
This method works extremely well if initialized sufficiently near the desired resonance. 
We initialize Newton's method by either (i) using a finite difference discretization of Eq. \eqref{1dSRP} to form a quadratic eigenvalue problem (QEP) which is solved using Matlab's \verb+polyeig+ command \cite{Tisseur:2001fk}, or (ii) using the $\omega$ computed at a previous optimization iteration. 
Other references that derive the transmission conditions at material discontinuities, including for two-dimensional, radially symmetric $n(\x)$ are  \cite{Yeh:1988fk,Yeh:1978fk}.

\paragraph{Optimization problem.} 
In Prop. \ref{prop:gradomega} and App. \ref{sec:variations}, we compute the  variation $\frac{\delta \omega}{ \delta n}$. Thus, a number of gradient-based methods are available to solve Eq. \eqref{eqn:optOneRes}.
In \cite{OptimizationOfScatteringResonances,KaoSantosa}, the authors use  gradient descent methods.
To solve Eq. \eqref{eqn:optOneRes}, we have applied a BFGS interior point method.

\subsection{Computational results}
\label{sec:compRes}

\begin{figure}[t!]
\begin{center}
\includegraphics[width=3in]{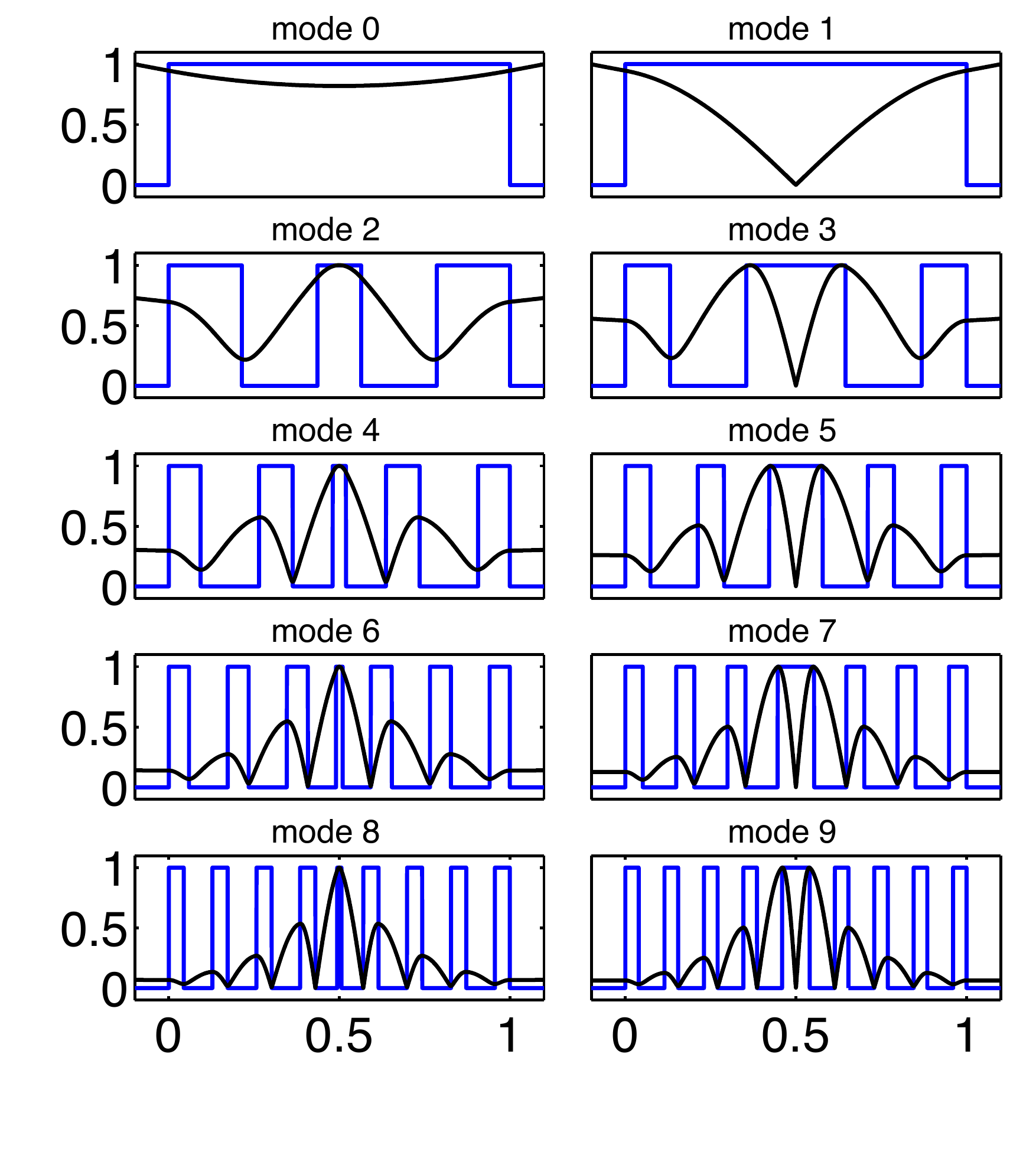} \ \
\includegraphics[width=2.5 in]{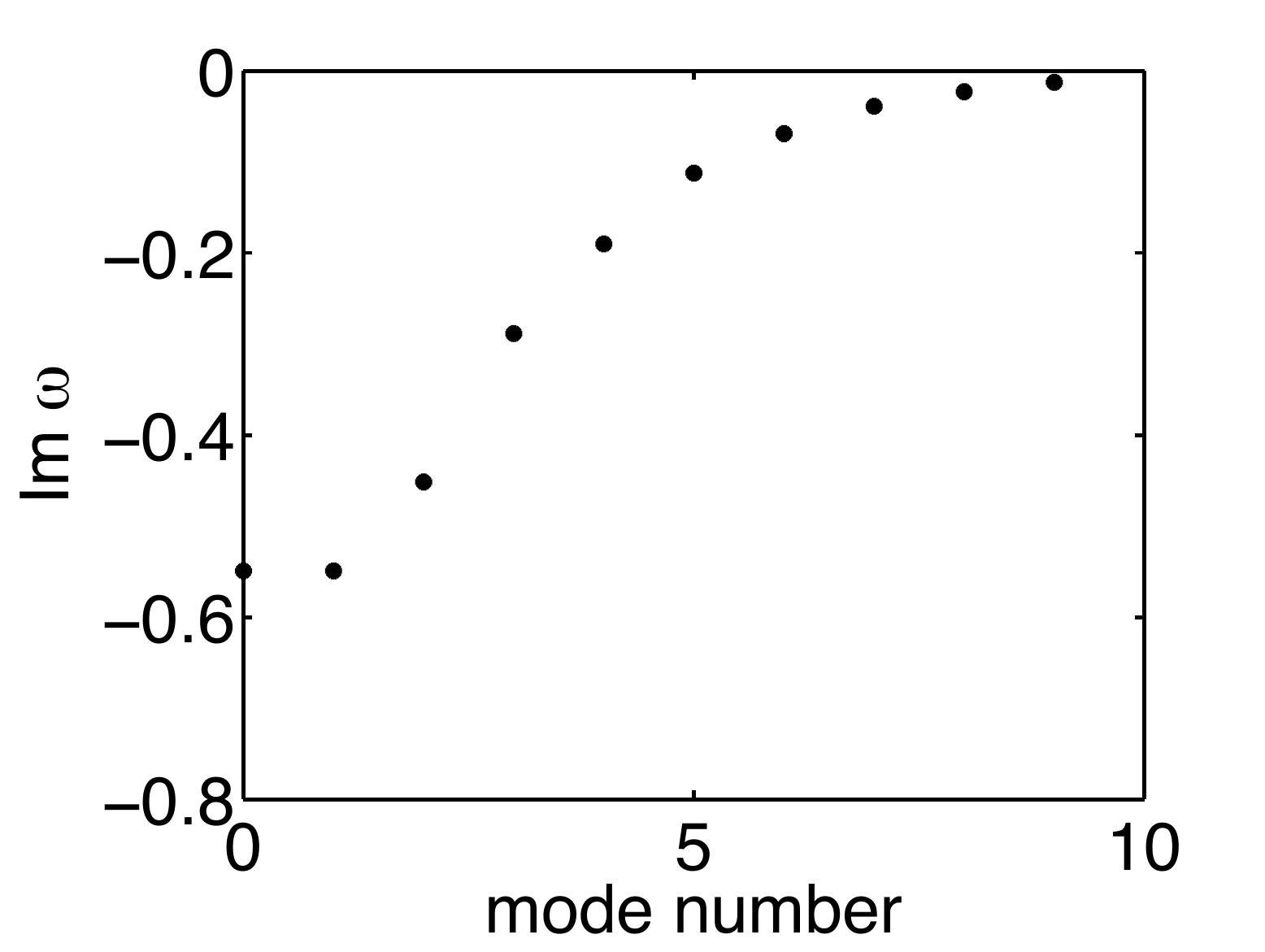}
\caption{Left plot: For $j=0,1,\ldots,9$, a plot of piecewise constant structure, $n^j_\star$, (blue) which  minimize \eqref{eqn:optOneRes2}. Corresponding mode modulus,  $|u^j_\star|$ (black)  is plotted
on the same axes. Right plot: Plot of $\Im \omega_\star^j$ for $j=0,1,\ldots,9$. }
\label{fig:OptModes}
\end{center}
\end{figure}

Let $\Omega = [0,L]$, $L=1$, $n_+=2$ and $n_-=1$. 
We define $\omega^j$, for $j=0,1,\ldots$, to be the scattering resonance for which the corresponding mode modulus has $j$ minima\footnote{except for mode $j=0$, whose modulus has one minima and whose eigenvalue  is purely imaginary} (see Fig.  \ref{fig:OptModes}). 
Using the method described in Sec. \ref{sec:CompMeth}, we solve
\begin{equation}
\label{eqn:optOneRes2}
n_\star^j := \arg \min_{n\in \calA} \ \   |\Im \omega^j[n] |. 
\end{equation}
 for $j=0,\ldots, 9$. For each $j$, we plot in 
Fig. \ref{fig:OptModes} the optimal refractive index $n_\star^j(x)$ and the corresponding resonance pairs denoted,
$(u_\star^j, \omega_\star^j)$. 
In Fig. \ref{fig:transCoef}, we plot the transmission coefficient modulus,  $|t^j(\omega)|$, associated with $n^j_\star(x)$. 
The transmission coefficient $t(\omega)$ is defined by the solution of the form
$$
u(x,\omega) = \begin{cases}
e^{\imath \omega x} + r(\omega) e^{-\imath \omega x} & x<0 \\
t(\omega) e^{\imath \omega x} &x>L
\end{cases}
\qquad \qquad \text{ for } \omega \in \mathbb R
$$

\begin{figure}[t!]
\begin{center}
\includegraphics[width=3.5in]{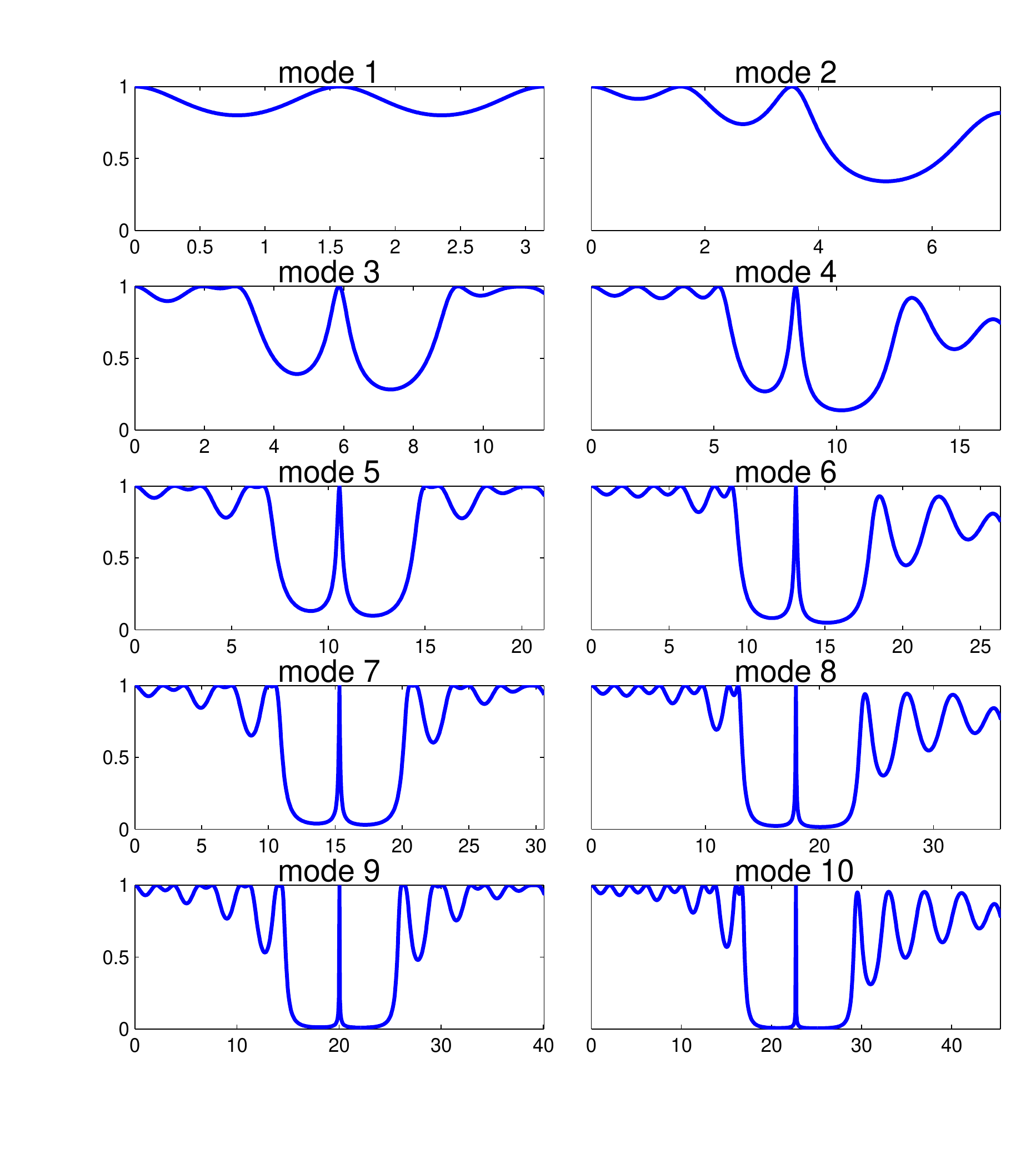}
\vspace{-.8cm}
\caption{The transmission coefficients $|t^j(\omega)|$ corresponding to the refractive indices $n_\star^j(x)$ minimizing Eq. \eqref{eqn:optOneRes2} for $j=1,\ldots,10$. The range of the $\omega$-axis is taken to be  $[0, \, 2 \Re \omega^j_\star ]$ to demonstrate that in the center, $|t^j(\Re \omega^j_\star)| \approx 1$, while $|t^j(\omega)|$ near $\omega=\Re \omega^j_\star$ is very small as $j$ increases.}
\label{fig:transCoef}
\end{center}
\end{figure}

\begin{rem}
We observe the following:
\begin{enumerate}
\item For each $j$, the optimal $n_\star^j$ is symmetric, \ie, $n^j_\star\in\calA_{sym}\subset \calA$.
\item For each $j$,  $n^j_\star(0)=n^j_\star(L/2)=n^j_\star(L)=n_+$. 
\item For each $j$, the optimal $n^j_\star(x)$ is roughly  periodic with a defect  near $x=L/2$. The number of repeated blocks increases as $j$ increases. The width of the repeated structure approximately satisfies 
Bragg's relation as defined in Prop. \ref{prop:Bragg}. 
As $j$ increases, for even modes, the defect width tends to zero, while for odd modes, the defect width tends to approximately twice the Bragg width.  These observations will be made more precise in Sec. \ref{sec:1d-bragg}.

\item The values $| \Im \omega_\star^j|$ are monotonically decreasing in $j$, \ie, $|\Im \omega_\star^{j+1}| <  |\Im \omega_\star^{j}|$.
\item The values $ \Re \omega_\star^j$ are monotonically increasing in $j$, \ie, the optimal modes $u_\star(x)$ becomes increasingly oscillatory as $j\uparrow \infty$. 
\end{enumerate}
\end{rem}

\section{Characterization of locally  optimal $n(x)$ in dimension one / Bragg relation}
\label{sec:1d-bragg}

In this section, we fix $n_+>1$, $n_-=1$, $\rho$, and a domain $[0,L]$. We will use a combination of numerical observations from section \ref{sec:1dnum} and analysis to characterize one-dimensional refractive indices $n_\star \in \calA$ which are  
\begin{equation}
\tag{\ref{eqn:designProb}}
\textrm{local minima of } \ \  \Gamma^\rho[n]\ \ \textrm{subject to}\ \   n\in\calA .
\end{equation}

 In this section, we only consider states such that $|\Re \omega_\star | < \rho$, that is $\omega_\star$ is an interior point on the constraint set. Additionally, we take $\rho$ sufficiently large such that by Lemma \ref{cor:ReOmegaPositive}, $|\Re \omega_{\star}| > 0$. 

\paragraph{Even / oddness properties of $n_\star$ and $u_\star$} We begin the discussion with even/oddness properties of the optimal refractive indices $n_\star(x)$ satisfying Eq. \eqref{eqn:designProb} and the corresponding  modes $u_\star(x)$. The following conjecture is supported by numerical experiments (see Sec. \ref{sec:compRes}). 
\medskip

\begin{conj}
\label{conj:symN}
There exists $n_\star \in \calA_{sym}$, as defined in \eqref{A3-admiss}, such that $\Gamma^\rho[n_\star] = \Gamma_\star^\rho(\calA)$. That is, the refractive index obtained by minimizing $\Gamma^\rho$ over the set $\calA$ can be taken to be  symmetric.
\end{conj}
\medskip

\noindent Motivated by this conjecture, in what follows we restrict consider optimization over the constraint set, $\calA_{sym}$, of symmetric structures.
 \medskip
 
\begin{prop}
\label{prop:uSym}
If $d=1$ and $n\in \calA_{sym}$, then any scattering resonance mode satisfying the Helmholtz equation \eqref{helm} is either even or odd with respect to $x=L/2$. 
\end{prop}
 \medskip

\begin{proof}
Let $(u(x),\omega)$ be a scattering resonance pair satisfying Eq. \eqref{1dSRP}. Since $n(L-x) = n(x)$, the following are also solutions:
\begin{align*}
u_1(x) &= u(x) + u(L-x),  \ \  \
u_2(x) = u(x) - u(L-x). 
\end{align*}
The Wronskian 
$
{\cal W}(u_1,u_2) \equiv u_1 u_2' - u_1' u_2,
$
 is constant in $x$. The outgoing boundary condition implies  $\mathcal{W}(x)\equiv 0$.  
Thus $u_1$ and $u_2$ are linearly dependent,  which is possible only if either $u_1$ or $u_2$ vanishes identically. Thus, $u(x)$ is either even or odd.
\end{proof}

From the numerical experiments in Sec. \ref{sec:compRes}, we observe that  modes $u_\star$ corresponding to locally optimal $n_\star$ may be either even or odd (see Fig. \ref{fig:OptModes}). It follows that the modes $u_\star$ corresponding to locally optimal solutions of \eqref{eqn:designProb} will be either even or odd depending on the choice of $\rho$, L, and $n_+$.  

\medskip

\begin{lem}
\label{rem:nonVanish}
Let $n(x)\in \calA_{sym}$, \ie\ $n(x)$ is symmetric about $L/2$. Then, if  $u(x)$ is a solution of the scattering resonance eigenvalue problem \eqref{1dSRP}, 
then neither  $u(x)$ nor  $u'(x)$ may vanish  for any $x \neq L/2$. 
\end{lem}
\begin{proof} Let $(u(x;\omega),\omega)$ denote a scattering resonance pair for \eqref{1dSRP} with 
$n(x)\in \calA_{sym}$. By Proposition \ref{prop:uSym}, we have $u(L/2;\omega)=0$ or $u'(L/2;\omega)=0$. Suppose there is a  
 $\xi\in \mathbb R\setminus\{L/2\}$ such that either $u(\xi;\omega)=0$ or $u'(\xi)=0$. Assume $\xi<L/2$. If $\xi>L/2$, the proof is similar.  Then, $u(x;\omega)$ is a solution of the eigenvalue problem value problem:
 \begin{align}
 &-\partial_x^2 u(x;\omega) =   \omega^2 n^2(x) u(x;\omega),\ \ \xi<x<L/2\label{sa-evp}\\
&\qquad \textrm{with one of the pairs of boundary conditions: } \nonumber\\
 & u(\xi)=0,\ \ u'(L/2)=0\ \ \ \textrm{or}\ \ u'(\xi)=0, \ \ u'(L/2)=0 \ \ \ \textrm{or} \nonumber \\
 & u(\xi)=0,\ \ u(L/2)=0\ \ \ \  \textrm{or}\ \ u'(\xi)=0, \ \ u(L/2)=0  \nonumber 
 \end{align}
 Multiplication of \eqref{sa-evp} by $\overline{u(x;\omega)}$, integration over the interval $[\xi,L/2]$, integration by parts and using the boundary conditions
  yields:
  \begin{equation}   \int_\xi^{L/2}|\partial_xu(x;\omega)|^2dx=
  \omega^2\ \int_\xi^{L/2}n^2(x)|u(x;\omega)|^2dx
 \label{integ-id} \end{equation}
 It follows that $\omega^2$ is real and positive. Therefore $\omega$ is real. However, Theorem \ref{thm:exist} precludes the existence of a real scattering resonance energy. Thus we have a contradiction and Lemma \ref{rem:nonVanish} is proved.
\end{proof}

\medskip

The following proposition describes the change in argument of a solution $u(x)$ to 
the Helmholtz equation \eqref{helm} with arbitrary symmetric index, $n(x)$. 
\medskip

\begin{prop}[Monotonicity of phase]
\label{prop:argu}
Let $\left(\omega, u(\cdot,\omega)\right)$ denote a scattering resonance  pair satisfying the Helmholtz equation \eqref{helm} with $n\in \calA_{sym}$, \ie, symmetric about $x=L/2$. Then
\begin{equation}
\label{eqn:symArg}
\frac{\ud }{\ud x} \arg u(x) =  2 (\Re \omega) |\Im \omega| |u(x)|^{-2} \int_{L/2}^x n^2(z) |u(z)|^2 \ud z. 
\end{equation}
If $\Re \omega >0$, the corresponding resonance state has increasing argument for $x>L/2$ and decreasing argument for $x<L/2$. The opposite statement holds true for the resonance state corresponding to $-\overline{\omega}$.
In addition, for $x\in [0,L]$, 
\begin{equation}
\label{eqn:argBnd}
\left| \frac{d}{dx} \arg u(x) \right| \leq  C_{\text{Lip}}.
\end{equation}
\end{prop}

\begin{proof}
A similar proof may be found in \cite{II:1986uq}. By Lemma \ref{rem:nonVanish}, $u(x)$ does not vanish except possibly at $x=L/2$. Thus on the interval $[L/2+\epsilon,\infty)$ for small $\epsilon>0$, we make the Ricatti transformation  $y = \frac{u'}{u}$ and note that $\Im y = \left( \Im (\log u) \right)' = (\arg u)' $. We derive a differential equation for $\Im y$ and  show that the solution has the desired properties. Using Eq. \eqref{1dSRPa} we obtain
$y' = - \omega^2 n^2 - y^2$, which has imaginary part:
$$
\left( \Im y \right)' = 2(\Re \omega)  |\Im \omega|  n^2 - 2  (\Re y)  (\Im y),
$$
where we use that $\Im\omega=-|\Im\omega|$ since scattering resonances lie in the lower half plane.
Using the integrating factor $\exp( \log |u|^2 ) = |u|^2$, we obtain the differential equation
$
\frac{\ud}{\ud x} \left[ |u|^2  \Im y \right] = 2(\Re \omega) |\Im \omega| n^2 |u|^2. 
$
Integrating from $\frac{L}{2} + \epsilon$ to arbitrary $x$ we obtain
\begin{align*}
\frac{\ud }{\ud x} \arg u(x) 
=  |u(x)|^{-2} \left( |u(\frac{L}{2} + \epsilon)|^2 \Im y(\frac{L}{2} + \epsilon) + 2 (\Re \omega) |\Im \omega| \int_{\frac{L}{2}+\epsilon}^x n^2(z) |u(z)|^2 \ud z \right)
\end{align*}
Taking the limit as $\epsilon\downarrow 0$, we find that  $|u(\frac{L}{2} + \epsilon)|^2 \Im y(\frac{L}{2} + \epsilon) \rightarrow 0$ if $u(x)$ is either even or odd giving \eqref{eqn:symArg} for $x \in[L/2, L]$. The proof is similar for $x\in(-\infty,L/2]$.  

If $u(x)$ is even about $x=L/2$, Eq. \eqref{eqn:argBnd} follows from Eq. \eqref{eqn:symArg}, Lemma \ref{lem:intn2u2}, and Lemma \ref{rem:nonVanish} ($u$ doesn't vanish). 
If $u(x)$ is odd about $x=L/2$, then $|u(L/2)|=0$ and one may use  L'H\^opital's rule to show that 
$$
\lim_{x\rightarrow L/2} |u(x)|^{-2} \int_{L/2}^x n^2(z) |u(z)|^2 \ud z = 0.
$$
\end{proof}

Proposition \ref{prop:argu} can be used to obtain a significantly more detailed characterization
 of one-dimensional local optima, $n_\star(x)$, than provided in the general results of  Prop. \ref{prop:bangbang} and Prop. \ref{corr:nChar}.   In particular, we have
\medskip

\begin{thm}
\label{prop:stepFun}
In spatial dimension one, the optimal refractive index, $n_\star(x) \in {\cal A}_{sym}$, has a finite number of transition points between $n_{+}$ and $n_{-}$ and is therefore a step function. 
\end{thm}
\medskip

\begin{proof} Our proof is based on the approach taken in \cite{Karabash:2011uq}. 
Suppose that the optimal refractive index, $n_\star(x)$, has an infinite number transition points, labeled $\{ x_{j}\}_{j=1}^{\infty}$, in the interval $[0,L]$. 
By Proposition \ref{corr:nChar}, at each transition point $x_{j}$, we have $\Im [\alpha_{\star}  \omega_\star^2 u_\star^2(x_{j}) ] = 0$,
 where $\alpha_{\star}$ is a nonzero constant depending on $u_{\star}(x)$ and $\omega_{\star}$ is in the open lower-half plane. 
Since the sequence $\{ x_{j}\}$ is bounded, there exists a convergent subsequence, relabeled $\{ x_{j}\}_{j=1}^{\infty}$, which converges to a point $X \in [0,L/2]$. We consider two cases.

\emph{Case 1:} 
Suppose $X\neq L/2$. 
By continuity of $u_\star(x)$,  the real sequence $\alpha_\star\omega_\star^2 u_\star^2(x_j)$ converges to  $\alpha_\star\omega_\star^2 u_\star^2(X)\in \mathbb R$. 
In addition, the real sequence of difference quotients
$\alpha_\star \omega_\star^2 [ u_\star^2(x_j) - u_\star^2(X) ] / (x_j-X)$ converges to 
$2\alpha_\star \omega_\star^2 u_\star(X)  u'_\star(X) \in \mathbb R$. 
Note that $u_\star(X)\neq 0$ (see Lemma \ref{rem:nonVanish}). 
Thus the ratio
$$
\frac{2 \alpha_\star \omega_\star^2 u_\star(X) u'_\star(X)}{\alpha_\star\omega_\star^2 u_\star^2(X)}
$$
 is a real number, say $2 \rho$, implying  $u'_{\star}(X)= \rho u_{\star}(X)$. Thus $u_{\star}(x)$ satisfies a self-adjoint boundary-value problem on the interval $[X,L/2]$, contradicting Lemma \ref{rem:nonVanish}. 

\emph{Case 2:} 
Suppose $X=L/2$. Thus $ \alpha_\star\omega_\star^2 u_\star^2(x_j)$ is a sequence of real numbers converging to $ \alpha_\star\omega_\star^2 u_\star^2(L/2)$. Equation \eqref{eqn:symArg} implies that $\arg u_{\star}(x)$ is monotonic on the interval $[0,L/2]$. Indeed for $\Re \omega_{\star} >0$, $\arg u_{\star}(x)$ is strictly monotonic. Thus 
$$
| \arg\left[  \alpha_{\star} \omega_{\star}^{2}u^{2}_{\star}(x_{j+1}) \right] - \arg \left[ \alpha_{\star} \omega_{\star}^{2} u^{2}_{\star}(x_{j})\right] | \geq 2 \pi.
$$
On the other hand, by Eq. \eqref{eqn:argBnd},  we have
$$
| \arg\left[  \alpha_{\star} \omega_{\star}^{2}u^{2}_{\star}(L/2) \right] - \arg \left[ \alpha_{\star} \omega_{\star}^{2} u^{2}_{\star}(x)\right] | \leq C_{\text{Lip}}  |L/2 - x|
$$
which is a contradiction.  
This completes the proof of Theorem \ref{prop:stepFun}.
\end{proof}
\medskip

Let an optimal refractive index, $n_\star(x)$, have $N+1$ discontinuity points, say $N+1$, which we denote by $\{ x_j \}_{j=0}^N$,  such that\footnote{Note: we have not ruled out the possibility that $x_0>0$ and $x_N<L$.} 
 \begin{equation}
 \label{eqn:defxj}
n_\star(x) = \begin{cases}
n_+ &\quad   x \in (x_j, x_{j+1}), \ \ j \text{ even} \\
1 &\quad  \text{otherwise.}
\end{cases}
\end{equation}
In Figure \ref{fig:OptSchem}, we illustrate the relationship between $n_\star(x)$ and the sign of 
 $\Im(\alpha\omega_\star^2u_\star^2(x))$ by plotting a locally optimal refractive index $n_\star(x)$ and the quantity $\Im [\alpha \omega_\star^2 u_\star^2(x) ]$. 
 
 \begin{figure}[t]
\begin{center}
\includegraphics[width=5in]{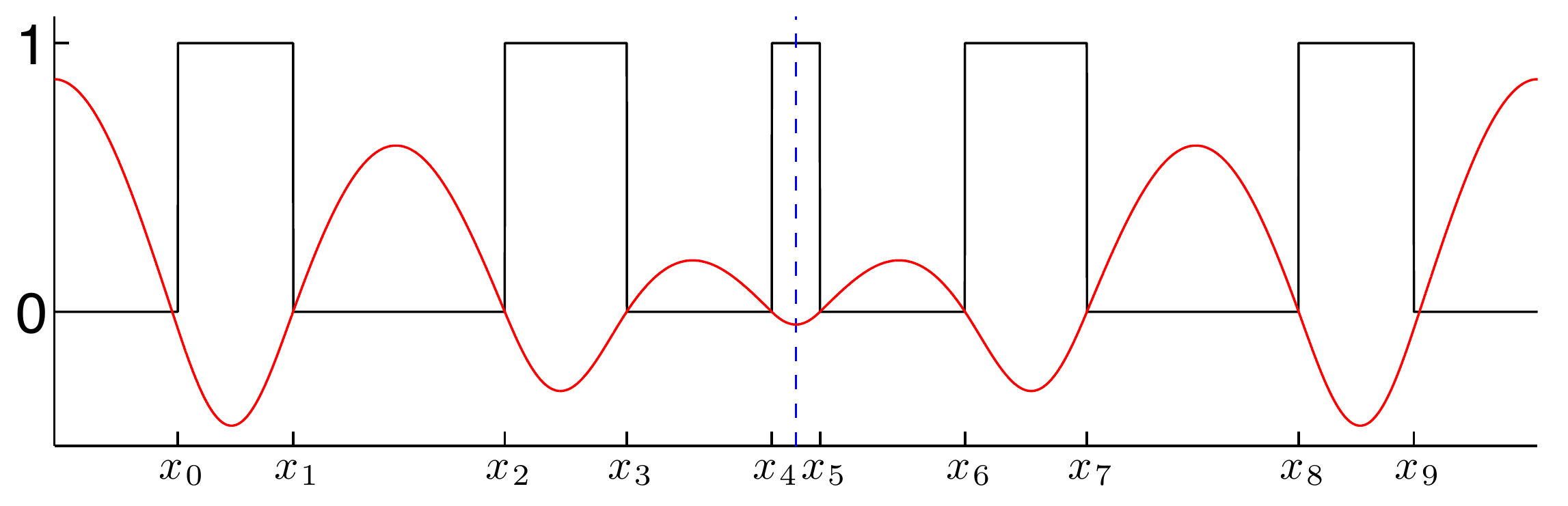}
\caption[The optimal refractive index and corresponding mode]
{A plot of an optimal $n_\star(x)-1$ in  black  and $\Im [\alpha \omega_\star^2 u_\star^2(x) ]$  in red. Note that $x_0=0$ and $\Im (\alpha \omega_\star^2 u_\star^2(0)) \neq 0$, but the optimality condition \eqref{eqn:nChar} is satisfied. This optimal structure and mode were obtained by solving Eq. \eqref{eqn:optOneRes2}  for $j=4$.  }
\label{fig:OptSchem}
\end{center}
\end{figure}

Since the energy of a mode concentrates where $n$ is large and we expect the energy to be concentrated in the center of the interval $\Omega = [0, L]$, we conjecture the following: 
\begin{conj}
\label{conj:centerN}
$n_\star(L/2) = n_+$.
\end{conj}

\noindent Conjectures \ref{conj:symN} and \ref{conj:centerN} imply that the number of intervals with $n=n_+$ is odd. Equivalently, the total number of intervals, $N=4M+1$ for $M\in \mathbb N$. For example, for the refractive index in Fig. \ref{fig:OptSchem}, $N=9$ and $M=2$. We refer to $[x_0,x_1]$ as the \emph{left-most interval}, $[x_{N-1},x_N]$ as the \emph{right-most interval}, $[x_{(N-1)/2}, x_{(N+1)/2}]$ as the \emph{center interval}. All other intervals are called \emph{interior intervals}. 

\begin{prop} 
\label{prop:intLengthBnd}
The length of an interior interval $I_j = [x_j, x_{j+1}]$ satisfies 
\begin{equation}
\label{eqn:intLengthBnd}
|x_{j+1} - x_j| \geq  \frac{\min_{x\in I_j} |u_\star(x)|^2} {|u_\star(0)|^2}  \ \  \frac{1}{4}  \frac{2\pi}{ |\Re \omega_\star|}. 
\end{equation}
\end{prop}

\begin{proof}
Corollary \ref{corr:nChar} implies that the argument of $u_\star(x)$  changes by exactly $\pi/2$ on an interior interval $I_j=[x_j, x_{j+1}]$. Using Eq. \eqref{eqn:symArg}, we compute 
\begin{align*}
\frac{\pi}{2} &=  2 |\Re\omega_\star| \ |\Im \omega_\star| \int_{x_j}^{x_{j+1}} |u_\star(x)|^{-2} \int_{L/2}^x n_\star^2(z) |u_\star(z)|^2 \ud z \ud x \\
&\leq 2 |\Re\omega_\star| \ |\Im \omega_\star| \int_{x_j}^{x_{j+1}} |u_\star(x)|^{-2} \int_{L/2}^{L} n_\star^2(z) |u_\star(z)|^2 \ud z \ud x \\
&=  |\Re\omega_\star| \ |u_\star(0)|^2 \int_{x_j}^{x_{j+1}} |u_\star(x)|^{-2} \ud x \
\leq  |\Re\omega_\star| \ |u_\star(0)|^2 \  |x_{j+1} - x_j|  \  \max_{x\in I_j} |u_\star(x)|^{-2} 
\end{align*}
where we used Eq. \eqref{eqn:tau} and the symmetry of $n(x)$ and $|u_\star(x)|$. Now by Lemma \ref{rem:nonVanish},  $|u_\star(x)|>0$ for all  $x \in I_j$, from which Eq. \eqref{eqn:intLengthBnd} follows. 
\end{proof}

\subsection{Locally optimal refractive indices and the  Bragg relation} 
\medskip

Denote $\sigma(x) := \Im (\alpha \omega_\star^2 u_\star^2(x))$.  From Proposition \ref{corr:nChar}, 
\begin{equation}
\sigma(x_{0}) = 0 \Leftrightarrow \ \arg u_\star(x_1) - \arg u_\star(x_0) = \frac{\pi}{2}. 
\label{ref:BraggSuffCond}
\end{equation}
We remark that  numerical experiments suggest that  $\sigma(x_{0}) \approx 0$ (see  Fig. \ref{fig:OptSchem}) and that $\sigma(x_0)\to0$ as $n_+$ or $L\to\infty$.  For the idealized limiting case we have:
\medskip

\begin{prop}[Bragg relation]
\label{prop:Bragg}
Assume a locally optimal mode satisfies \eqref{ref:BraggSuffCond}. Then, 
 the intervals on which $n_\star(x)$ is constant have the following properties:
\begin{enumerate}
\item  The length of all non-center intervals, \ie, $I_j = [x_j, x_{j+1}]$ for all $j \neq(N-1)/2$  is given by
$d_\pm = \frac{1}{4}  \frac{2 \pi}{n_\pm | \Re \omega_\star|} = \frac{1}{4} \lambda_{\text{eff},\pm}$,
where $\lambda_{\text{eff},\pm}$ is the ``effective'' wavelength of $u_\star$ in the medium $n_\pm$. 
\item Thus, a full ``period'' of the repeated structure is 
$\delta = d_+  +   d_-  =   \frac{1}{2 n_h} \frac{2 \pi}{|\Re \omega_\star|}$,
where $n_h = 2(n_1^{-1} + n_2^{-1})^{-1}$ is the harmonic mean of $n_{+}$ and $n_{-}$.
\item Additionally, the length of the center interval is less than $2 d_+$.
\end{enumerate}
\end{prop}
\medskip

\noindent Since the numerical evidence is that \eqref{ref:BraggSuffCond} is satisfied only asymptotically for large $n_+$ or $L$, this suggests the following:
\medskip

\begin{conj}
\label{conj:Bragg}
The widths of the non-center intervals,
$$
\delta_{\pm}(n_{+}, n_{-}, L) \rightarrow \frac{1}{4} \frac{2 \pi}{ |\Re \omega_{\star}| n_{\pm}} = \frac{1}{4} \lambda_{\text{eff},\pm} 
\qquad \text{as $n_{+}\rightarrow \infty$ or $L \rightarrow \infty$}.
$$
\end{conj}

\begin{proof}[Proof of Proposition \ref{prop:Bragg}] Write $\omega_\star = \omega_R + \imath \omega_I$. 
The solution on any interval $I_{j} = [x_j, x_{j+1}]$ where $n_{\star}(x) \equiv n_\pm$ can be written
\begin{equation}
\label{eq:uansatz}
u_{\star}(x) = \alpha_{j} \cos[\omega_{\star} n_\pm (x-x_j)] - \imath \beta_{j} / n_\pm \sin[\omega_{\star} n_\pm (x-x_j)]
\end{equation}
where $\alpha_{j}$ and $\beta_{j}$ are chosen such that 
$\alpha_{j}=u_{\star}(x_j)$  and $ \beta_{j} = (\imath/\omega) u_{\star}'(x_j)$. 
We record the following formulas for use in the argument below:
\begin{align}
\nonumber
&u_{\star}'(x) = -\alpha_{j} \omega_{\star} n_\pm \sin[\omega_{\star} n_\pm (x-x_j)] - \imath \omega_{\star} \beta_{j} \cos[\omega_{\star} n_\pm (x-x_j)] \\
\nonumber
&\sin \left( \frac{\pi \omega_{\star}}{2 \omega_R}\right) = \cosh \left( \frac{\pi \omega_I}{2 \omega_R}\right), \quad
\cos\left( \frac{\pi \omega_{\star}}{2 \omega_R}\right) = -\imath \sinh \left( \frac{\pi \omega_I}{2 \omega_R}\right), \\
\label{eq:uprimec}
&\Rightarrow \quad  u_{\star}'(x_j + d_\pm) = -\alpha_{j} \omega_{\star} n_\pm \cosh \left( \frac{\pi \omega_I}{2 \omega_R}\right)  - \omega_{\star} \beta_{j} \sinh  \left( \frac{\pi \omega_I}{2 \omega_R}\right). 
\end{align}

Without loss of generality, let  $\Re \omega_{\star}>0$ and assume that on the first interval, $I_0 = [x_0, x_1]$,  $\alpha_{0} = \beta_{0} = 1$ which satisfy the outgoing boundary conditions for $x<x_{0}$. We now show that as $x$ increases, there are two alternating cases which define $n_{\star}(x)$ on the interior intervals. 

\paragraph{Case 1} Suppose that on the interval $I_j=[x_j, x_{j+1}]$, $\alpha_{j},\beta_{j} \in \mathbb R$ and $n=n_+$. This is the case for the first interval, $I_0$. 
On an interior interval, by Prop. \ref{corr:nChar}, the point $x_{j+1}$ is that point at which $\arg u_\star(x)$ has increased by $\pi/2$. On the first interval, this is precisely assumption  \eqref{ref:BraggSuffCond}.
Since $\Re u_{\star}(x_{j})=0$,  an increase of $\arg u_\star(x)$ by $\pi/2$ on $I_{j}$ is equivalent to $\Re u_{\star}(x_{j+1})=0$. 
The value of $x$ such that $\Re u_{\star}(x)$ given in \eqref{eq:uansatz} next vanishes is  $x_{j+1} = x_j + d_+$.  Using \eqref{eq:uprimec},  $u'(x_{j+1}) = \omega \gamma$ where $\gamma \in \mathbb R$. 
 Thus, the coefficients  $\alpha_{j+1}$ and $\beta_{j+1}$ on the next interval, $I_{j+1}$, are purely imaginary.

\paragraph{Case 2} Suppose that on an interior interval, $I_{j} =[x_j, x_{j+1}]$,  $n_{\star}(x)=n_-$ and $\alpha_{j}$ and $\beta_{j}$ are purely imaginary, \ie, $\alpha_{j},\beta_{j} \in \imath \mathbb R$. 
By Prop. \ref{corr:nChar}, the  point $x_{j+1}$ is defined to be the first value for which $\Im u_{\star}(x_{j+1})=0$. 
 Using, \eqref{eq:uansatz}, we find that $x_{j+1} = x_j + d_-$. 
By   \eqref{eq:uprimec}, $u'(x_{j+1}) = \imath \omega \gamma$ where $\gamma \in \mathbb R$. Thus the coefficients  
$\alpha_{j+1}$ and $\beta_{j+1}$ on the next interval $I_{j+1}$ are purely real, which is precisely Case 1.

Thus, the intervals alternate between Cases 1 and 2 until the center interval, $I_{(N-1)/2} = [x_{(N-1)/2}, x_{(N+1)/2}]$.  is reached. Here, by Prop. \ref{prop:argu}, $\arg u_{\star}(x)$ increases on $[x_{(N-1)/2}, L/2]$ and decreases on $[L/2, x_{(N+1)/2}]$. 
By Prop. \ref{corr:nChar},  $\arg u_{\star}(x)$ changes by less than $\pi/2$ on each of these half intervals, which implies that the width of each of these half intervals is less than $d_+$. 
\end{proof}

\section{Discussion}
In this article we have investigated the question of finding compactly supported structures, defined in terms of a refractive index, $n(\x)$,  for which the Helmholtz equation has very long lived scattering resonances in dimension d =1,2,3. Here $n$ varies in an admissible set of refractive indices ${\cal A}\left(\Omega,n_-,n_+,\right)$,
defined in terms of the support of the structure ($\Omega$) and upper ($n_+$) and lower ($n_-$) pointwise bounds for $n$.  We now summarize,  based on our analytic, asymptotic, and numerical studies, further expectations, beyond what we've established, for the optimal one-dimensional structure and its corresponding resonance of minimal imaginary part,  $|\Im\omega_\star|\equiv\Gamma_\star^\rho$,  

\begin{enumerate}
\item We expect that the optimal resonance $\omega_\star$ such that $\Gamma_\star^\rho(\calA) = |\Im \omega_\star|$ is achieved for a resonance with real part approximately imposed upper bound:
$\Re \omega_\star \approx \rho$.
We also expect its imaginary part to be exponentially small in $\rho$: $\Im \omega_\star \sim A \exp(- B \rho)$,
where $A$ and $B$ are constants dependent on $\cal A$.
\item Furthermore, we expect that the number of times  the optimal refractive index, $n_\star(x)$, alternates between $n_+$  and $n_-$  on $[0,L/2]$ is approximately $M = (N-1)/4  \approx \frac{L/2}{\delta} = \frac{n_h \Re\omega L}{2\pi} \approx \frac{n_h \rho L}{2\pi}.$
\item Simple models \cite{Heider-Weinstein-unpubl}  suggest  that $|\Im \omega_\star| \sim n_+^{-M}$
and thus $\Gamma_\star^\rho \sim n_+^{-\rho n_h L/ 2\pi}$. 
\end{enumerate}
\medskip

Finally, we remark on a number of interesting future questions and directions.
In the present work we have established that classical one-dimensional Bragg structures arise as optimal one-dimensional compactly supported dielectric cavities with finite material contrast in the limit that the support or material contrast tends to infinity. Future work will include study of the deviation of locally optimal $n_\star(x)$ from the Bragg relation defined in \eqref{prop:Bragg} for finite support and material contrast. 

Our work establishes the existence of a solution to the optimal  design problem in two and three spatial dimensions ($d=2,3$). It would be interesting to study the limit of large support or large material contrast for $d=2$ or $3$. Perhaps this leads to a  higher-dimensional variant of Bragg structures. 

We conclude mentioning that the measure of resonance lifetime to be optimized is often chosen to suit the particulars of the problem. We have chosen to maximize: $\tau = |\Im \omega|^{-1}$. Other choices  include the quality factor, $Q = \frac{\Re \omega}{2 |\Im \omega |}$, which measures loss per unit oscillation cycle,  and the Purcell factor, $\propto Q/V$, where $V$ denotes the ``mode volume''. The Purcell factor is particularly important in applications where a strong light-matter interaction is desired. Maximization of such  quantities would be a natural and interesting extension of this work.

\appendix

\section{Examples: resonances for radially symmetric cavities in $\mathbb{R}^d,\ d=1,2,3$}
\label{sec:simple-example}
In this subsection, we discuss the resonances for a simple example in dimension $d=1$, 2, and 3. We take 
$\Omega=\{\x\colon  |\x|<a \}$ and $n(\x)$ defined by: 
\begin{equation}
\label{eqn:nSimpEx}
n(\x)  = \begin{cases}
n_0 & |\x| <a \\
1 & |\x| >a.
\end{cases}
\end{equation}
where $n_0>1$ and $a>0$ are constants.
We now consider the scattering resonance problem  \eqref{SRP}; see also \cite[Ch. 8]{Nockel:1997fk} and \cite{Dubertrand:2008fk}.

\subsection{Dimension d=1}
Imposing outgoing radiation conditions we find that
\begin{equation}
\label{eqn:1DresState}
u(x) = \begin{cases}
A e^{-\imath \omega x} & x < -a \\
B e^{\imath \omega n_0 x} + C e^{-\imath \omega n_0 x} & |x| < a \\
D e^{\imath \omega x} & x > a \  .
\end{cases}
\end{equation} 
Imposing continuity of $u$ and $\partial_x u$ at $x=\pm a$ yields a  $4\times 4$ linear system of equations for the constants $A$, $B$, $C$, and $D$, whose non-trivial solvability requires the vanishing of a determinant. Resonances are the values $\omega \in \mathbb C$ for which this determinant vanishes:
$e^{-2 \imath \omega (n_0-1) a } \left[ e^{4 \imath  \omega n_0 a } (n_0-1)^2 - (n_0+1)^2\right] = 0, $
whose solutions  are given by 
\begin{equation}
\label{eqn:1Dres}
\omega_m = \frac{\pi m}{2 n_0 a} - \imath \frac{1}{2 n_0 a} \log \left| \frac{n_0 + 1}{n_0 - 1} \right|  \quad \quad m \in \mathbb N. 
\end{equation}
The resonances in Eq. \eqref{eqn:1Dres} for $n_0=2$, $a=1$ are plotted in Fig. \ref{fig:nDres}(left).

\begin{figure}[t]
  \begin{center}
  \includegraphics[width=2.05in]{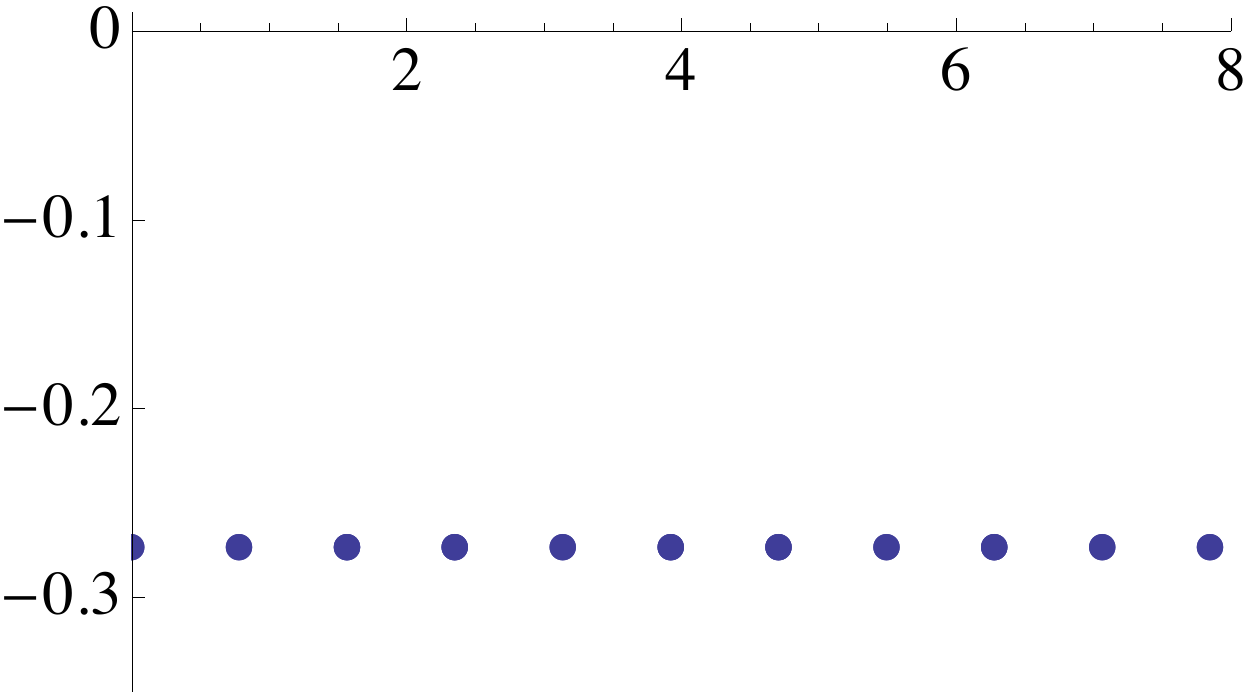}
  \includegraphics[width=2.05in]{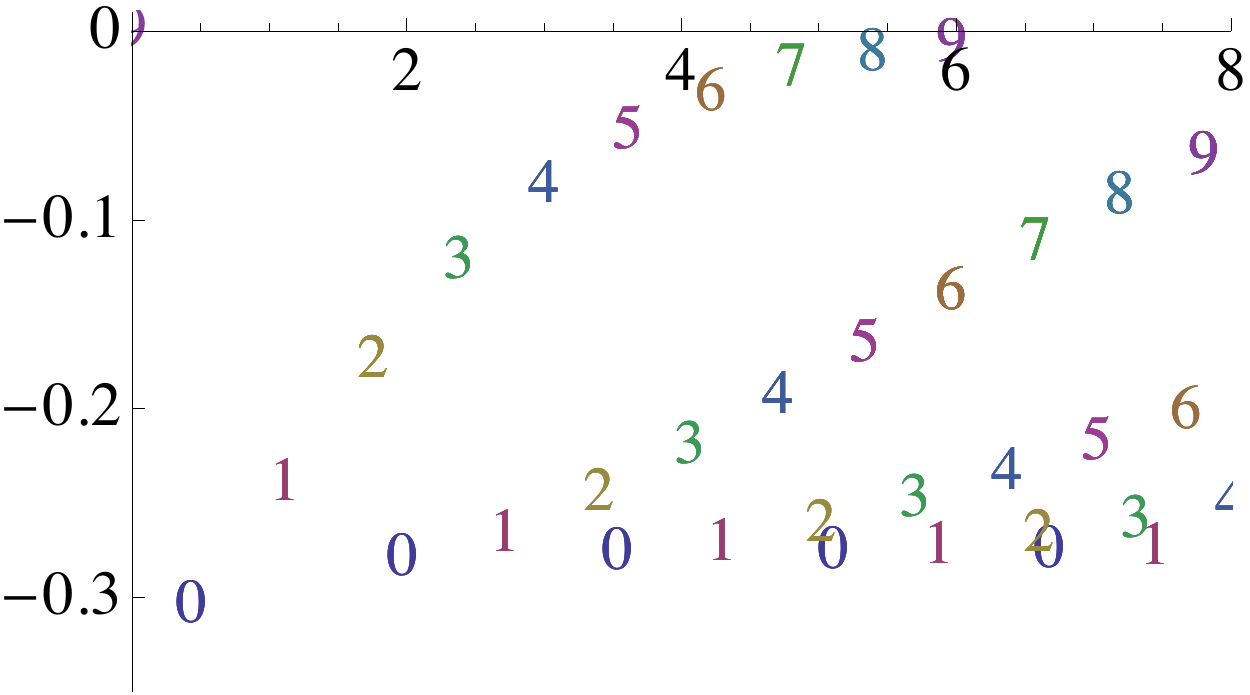}
  \includegraphics[width=2.05in]{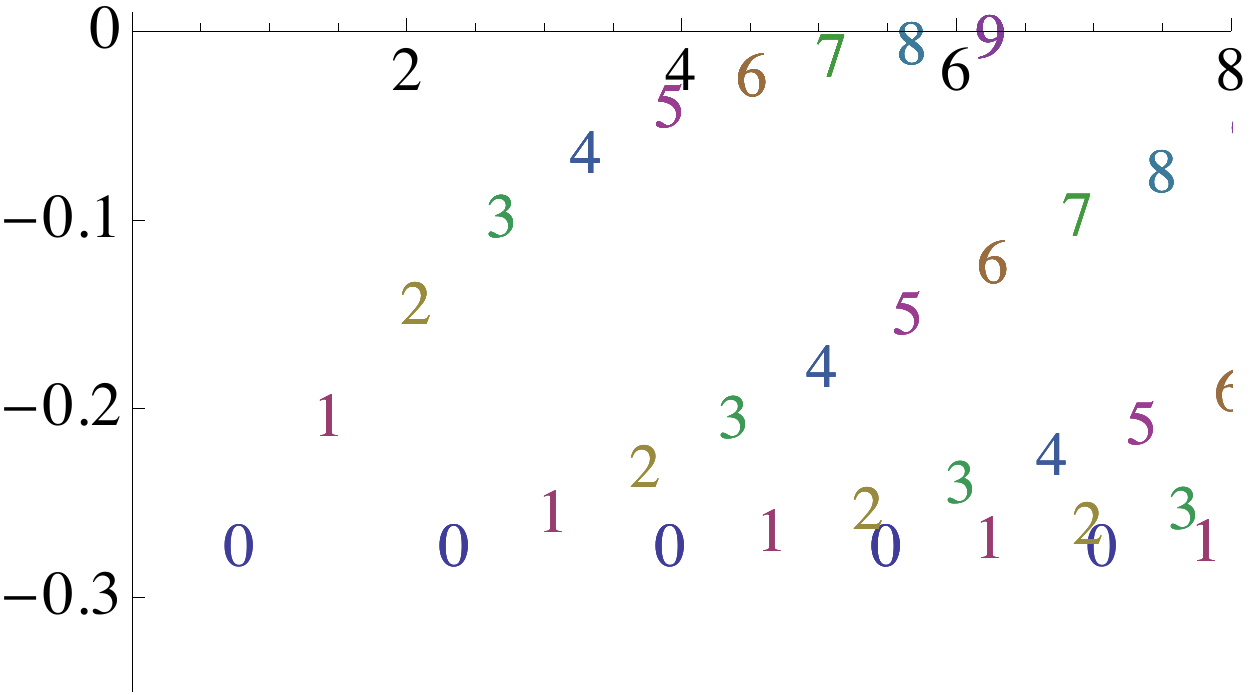}
  \caption{Helmholtz resonances for index of refraction given in Eq. \eqref{eqn:nSimpEx} in 1-, 2-, and 3-dimensions computed using Eqs. \eqref{eqn:1Dres}, \eqref{eqn:2Dres}, and \eqref{eqn:3Dres}. In each dimension, $n_0=2$ and $a=1$. \label{fig:nDres} }
  \end{center}  
\end{figure}

\subsection{Dimension d=2}
Outgoing solutions bounded at the origin are given by
$$
u(r,\theta) = \begin{cases}
A J_\ell(n_0 \omega r) e^{\imath \ell \theta} & r<a \\
B H_\ell^{(1)} (\omega r) e^{\imath \ell \theta} & r>a, \ \ell \in \mathbb Z
\end{cases}
$$
Imposing continuity of  $u$ and $\partial_r u$ at $r=a$ leads to a system of two linear homogeneous equations for $A$ and $B$. This equation is solvable if and only if
\begin{equation}
\label{eqn:2Dres}
J_\ell(n_0 \omega a)  H_\ell^{(1)\prime} (\omega  a)  - n_0  H_\ell^{(1)} (\omega  a) J_\ell'(n_0 \omega a) = 0. 
\end{equation}
We numerically solve Eq. \eqref{eqn:2Dres} for $n_0=2$, $a=1$, and $\ell=0,\ldots,9$ and plot the resonances in Fig. \ref{fig:nDres}(center). For each angular momentum, $\ell\ge0$, there is an infinite sequence of resonances, $\{\omega_{\ell,j}\}_{j\ge1}$, in the lower half plane. Fixing $\ell$ and letting $j$ tend to infinity, these resonances are seen to have the asymptotics:
\begin{equation}
\label{eqn:2dhighFreqAsym}
\omega_{\ell,j} \sim \frac{\pi}{n_0 a} \left( j + \frac{1}{4} + \frac{\ell}{2} \right) 
- \imath \frac{1}{2 n_0 a} \log \left| \frac{n_0 + 1}{n_0 - 1} \right|  \quad \quad j >> 1, 
\end{equation}
as derived in \cite{Nockel:1997fk}. For fixed $j$ and angular momentum, $\ell$, tending to infinity the sequence of resonances approaches the real axis exponentially  \cite{Nockel:1997fk,Dubertrand:2008fk}.  Note that $\omega_{0,j}$ has multiplicity $1$ and  $\omega_{\ell,j}$ has multiplicity $2$  for $j\ge1$. The mode plotted in Fig. \ref{fig:mech}(center) corresponds to the 2-dimensional resonance with $\ell=6$ and $j=1$.

\subsection{Dimension d=3} 
The analysis is analogous to the case of spatial dimension $2$. 
Solutions bounded at the origin and outgoing at infinity are given by 
$$
u(r,\theta,\phi) = \begin{cases}
A j_\ell (n_0 \omega r)  Y_\ell^m (\theta,\phi) & r<a \\
B h_\ell^{(1)} (\omega r) Y_\ell^m (\theta,\phi) & r>a,\ \ |m| \leq \ell \in \mathbb N\ .
\end{cases}
$$
Imposing continuity of  $u$ and $\partial_r u$ at $r=a$ leads to the solvability condition 
\begin{equation}
\label{eqn:3Dres}
j_\ell(n_0 \omega a)  h_\ell^{(1)\prime} (\omega  a)  - n_0  h_\ell^{(1)} (\omega  a) j_\ell'(n_0 \omega a) = 0. 
\end{equation}
We numerically solve Eq. \eqref{eqn:3Dres} for $n_0=2$, $a=1$, and $\ell=0,\ldots,9$ and plot the resonances in Fig. \ref{fig:nDres}(right). For fixed angular momentum,  $\ell$,
 and $j$ tending to infinity, the scattering resonances follow the asymptotics
\begin{equation}
\label{eqn:3dhighFreqAsym}
\omega_{\ell,j} \sim \frac{\pi}{n_0 a} \left( j + \frac{1}{2} + \frac{\ell}{2} \right) 
- \imath \frac{1}{2 n_0 a} \log \left| \frac{n_0 + 1}{n_0 - 1} \right|  \quad \quad m >> 1. 
\end{equation}
For fixed $j$ and large angular momentum, $\ell$,  an asymptotic analysis involving spherical Hankel functions, analogous to that in \cite{Nockel:1997fk,Dubertrand:2008fk}, shows that the $\omega_{\ell,j}$ tend exponentially toward the real axis. Note that $\omega_{\ell,j}$ has multiplicity $2\ell+1$.

\section{Calculation of the variation  of a scattering resonance, $\omega$, with respect to the refractive index, $n(x)$; proof of Proposition \ref{prop:gradomega}}
\label{sec:variations}

The scattering resonance problem \eqref{SRP} can be written 
\begin{align}
\label{SRP2}
u(\x;\omega) =  \omega^2  R_\omega[u](\x;\omega ),  \text{ where } R_\omega[u](\x;\omega ) := \int_{\Omega} G(\x,\y ,\omega) \ m(\y) \ u(\y;\omega ) \ud \y
\end{align}
and $m(\x) := n^2(\x) -1$ is non-negative with support in $\Omega$. 

 We consider a perturbation of $n\in \calA$ of the from $n(\x) \mapsto n(\x) + \delta n(\x)$ where $\delta n(\x) \in L^\infty(\mathbb R^d)$ with support in $\Omega$. Denote variations with respect to $n$ using a prime and the variation of $G(\x,\y,\omega)$ with respect to $\omega$ by $\dot{G}(\x,\y,\omega)$. For $\x \in \Omega$,
we take variations of \eqref{SRP2} to obtain 
\begin{align*}
\left( \text{Id} - \omega^2 R_\omega \right) u'(\x;\omega) 
&=  2 \omega \omega' R_\omega [u](\x;\omega)  
+ \omega^2 \int_{\Omega} \dot{G}(\x, \y,\omega)\ \omega' \ m(\y) \ u(\y;\omega ) \ud \y \\
& \quad + \omega^{2} \int_{\Omega} G(\x, \y,\omega) \ m'(\y) \ u(\y;\omega ) \ud \y.
\end{align*}
Denoting  $\langle f, g \rangle = \int_\Omega \overline{f} g $ and  taking the inner product with a general $v \in L^2(\Omega)$, we obtain 
\begin{align}
\label{eqn:dotV}
\langle v, \left( \text{Id} -  \omega^2 R_\omega \right) u' \rangle =&   2 \omega \omega' \langle v,  R_{\omega}[u](\cdot;\omega)\rangle \\
\nonumber
& + \omega^{2} \omega' \int_{\Omega} \int_{\Omega} \overline{v(\x)} \ \dot{G}(\x,\y,\omega) \ m(\y) \ u(\y;\omega) \ud \y \ud \x \\
\nonumber
 & + \omega^{2} \int_{\Omega} \int_{\Omega}  \overline{v(\x)} \ G(\x,\y,\omega) \ m'(\y) \ u(\y;\omega)\ud \y \ud \x.
\end{align}
Using the adjoint operator 
$$
R_\omega^\ast[v](\x;\omega) = m(\x) \int_{\Omega} \overline{G (\x,\y,\omega)} v(\y) \ud \y,
$$ 
the left hand side of \eqref{eqn:dotV} can be rewritten  
\begin{align*}
\langle v, \left( \text{Id} - \omega^2 R_\omega \right) u' \rangle
&= \int_\Omega \overline{v}(\x) u(\x) \ud \x - \omega^{2} \int_\Omega \int_\Omega \overline{v(\x)} G(\x,\y,\omega) \ m(\y) \ u'(\y;\omega)\ud \y \ud \x \\
&= \langle \left( \text{Id} - \overline{\omega}^2  R_\omega^{\ast} \right) v,  u' \rangle. 
\end{align*}
Setting $v(\x) = v(\x; \omega) = \overline{\omega}^{2} m(\x) \overline{u(\x; \omega)}$, we compute

\begin{align}
\left( \text{Id} - \overline{\omega}^2  R_\omega^{\ast} \right) v(\x; \omega) 
&=   \overline{\omega}^{2} m(\x) \left( \overline{u(\x;\omega)} - \overline{\omega}^2 \int_{\Omega} \overline{G (\x,\y,\omega)} m(\y) \overline{u(\y;\omega)} \ud \y \right)\nonumber \\
&=  \overline{\omega}^{2} m(\x) \overline { \left( \text{Id}  - \omega^2 R_\omega  \right) u(\x; \omega) }\ = 0 .
\label{eq:EqForV}\end{align}
Therefore, for the above choice of $ v(\x;\omega)$, the left hand side of \eqref{eqn:dotV} vanishes.
The first and third terms on the right hand side of \eqref{eqn:dotV} are evaluated to be:
\begin{align}
\label{piece1}
\left\langle v,  R_{\omega}[u] \right\rangle = \left\langle \overline{\omega}^{2} m\overline{u},  \omega^{-2} u \right\rangle = \left\langle  m \ \overline{u},  u \right\rangle 
\end{align}
and
\begin{align}
\label{piece2}
 \int_{\Omega} \int_{\Omega}  \overline{v(\x)} \ G(\x,\y,\omega) \ m'(\y) \ u(\y;\omega)\ud \y \ud \x = \langle \overline{u}, m' u \rangle.
\end{align}
Therefore \eqref{eqn:dotV} reduces to
\begin{align*}
\omega' \left( 2 \omega  \langle \overline u, m \ u \rangle + \omega^{4} \int_{\Omega} \int_{\Omega} m(\x) \ u(\x;\omega) \ \dot{G}(\x,\y,\omega) \ m(\y) \ u(\y;\omega) \ud \y \ud \x \right)   =
  - \omega^{2} \langle \overline{u},m' \ u \rangle 
\end{align*}
or, equivalently, recalling that $m'=(n^2)'=2n\delta n$, 
\begin{align*}
\omega' =  - 2 \alpha \omega^{2} \left\langle n \overline{u(\cdot; \omega)}^2 , \delta n \right\rangle \equiv \left\langle \frac{\delta \omega}{\delta n}, \delta n \right\rangle,
\end{align*}
where
$$
\alpha^{-1} =  2 \omega  \langle  \overline {u(\cdot; \omega)}, m \ u(\cdot; \omega) \rangle + \omega^{4} \int_{\Omega} \int_{\Omega} m(\x) \ u(\x ; \omega) \ \dot{G}(\x,\y,\omega) \ m(\y) \ u(\y;\omega) \ud \y \ud \x   .
$$
Thus, 
$$
\left| \omega[n+\delta n] - \omega[n] - \left\langle \frac{\delta \omega}{\delta n} , \delta n\right\rangle \right| = o\left( \| \delta n \|^2_{L^2(\Omega)} \right)
$$
with
\[ \frac{\delta \omega}{\delta n}\ =\ - 2\ \overline{\alpha}\ \overline{\omega}^{2}  n \overline{u(\cdot; \omega)}^2, \]
proving \eqref{eqn:gradomega}. 

In one dimension,  the explicit formula for the Green's function given in  \eqref{eqn:explicitG-1} may be used to obtain the first equality in \eqref{eqn:1dalpha}. This formula agrees with \cite[Eq. (19)]{OptimizationOfScatteringResonances}. The second equality in \eqref{eqn:1dalpha} can be obtained using the identity
$
\omega^2 \int n^2 u^2 =  \int u_x^2 - \imath \omega [u^2(0) + u^2(L)],
$
obtained by multiplying \eqref{1dSRPa} by $u$ and integrating by parts.

\section{Bragg's relation maximizes the spectral gap to \\ midgap ratio} \label{sec:Bragg}
Consider a one-dimensional $n(x)$, infinite in extent, which is periodic with period $d$, \ie, $n(x+d) = n(x)$, and has alternating layers,  \ie,
$$
n(x) = 
\begin{cases} 
n_1 & 0< x< b \\
n_2 & b< x< d.
\end{cases}
$$
It is shown in \cite{Yeh:1988fk} that the wave equation \eqref{eqn:wave} has a solution of the form $v(x,t) = e^{\imath (\omega t - kx)} u_{per}(x)$ where $u_{per}(x+d) = u_{per}(x)$ if $\omega$ and $k$ satisfy the dispersion relation
\begin{equation}
\label{eqn:dispRel}
cos(kd) = cos(\omega n_1 b) \cos \omega n_2 (d-b) - \frac{1}{2}\left( \frac{n_2}{n_1} + \frac{n_1}{n_2}\right) \sin(\omega n_1 b) \sin \omega n_2 (d-b).
\end{equation} 
The \emph{Bragg relation} is defined
\begin{align}
\label{eqn:Bragg}
n_1b =  n_2 (d-b)  =\frac{1}{4} \frac{2 \pi}{\omega} 
\quad \Rightarrow \quad
d= \frac{1}{2 n_h} \frac{2 \pi}{\omega}, \quad b= \frac{n_h}{n_1} \frac{d}{2},  \quad (d-b) = \frac{n_h}{n_2} \frac{d}{2}
\end{align}
where $n_h = 2(n_1^{-1} + n_2^{-1})^{-1}$ is the harmonic mean of $n_1$ and $n_2$. The width of each layer is a quarter wavelength, \ie, the phase of the wave changes by $\pi/2$ in each layer. 
For a Bragg structure, the dispersion relation \eqref{eqn:dispRel} simplifies: 
$$
cos(kd) = 1 - \frac{1}{2} \frac{(n_1+n_2)^2}{n_1 n_2} \sin^2\left(\frac{\omega n_h d}{2}\right).
$$
Thus, the dispersion relation for a Bragg structure has a spectral band gap centered at $\omega = \frac{\pi}{d n_h}$ of width $\frac{4}{d n_h} \sin^{-1}\frac{|n_1 - n_2|}{n_1 + n_2}$. The dispersion relation with  $n_1=1$ and $n_2=2$ is plotted in Fig. \ref{fig:bragg}(left). 

\begin{figure}[t]
\begin{center}
\includegraphics[width=2.5in]{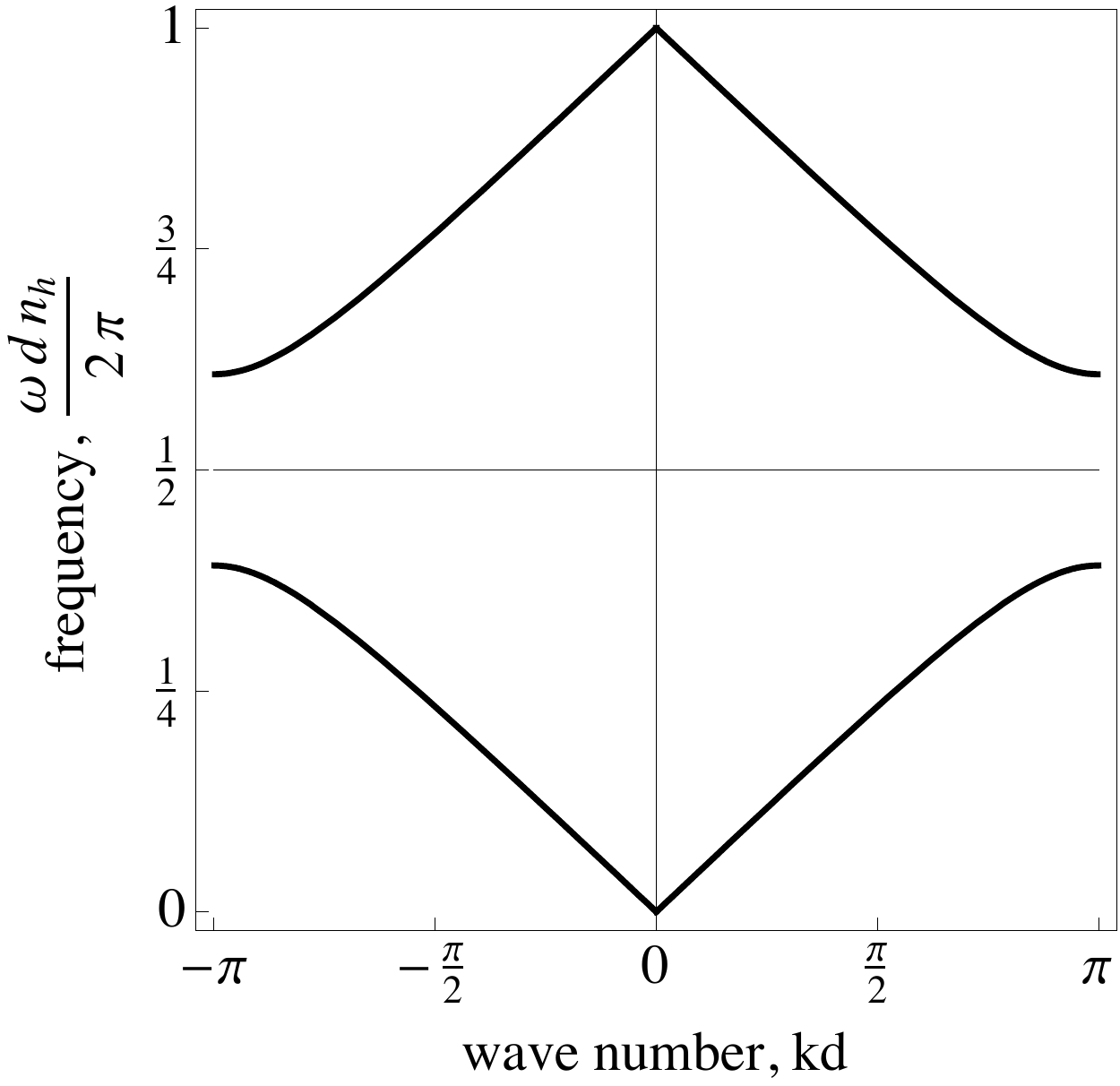}
\includegraphics[width=2.5in]{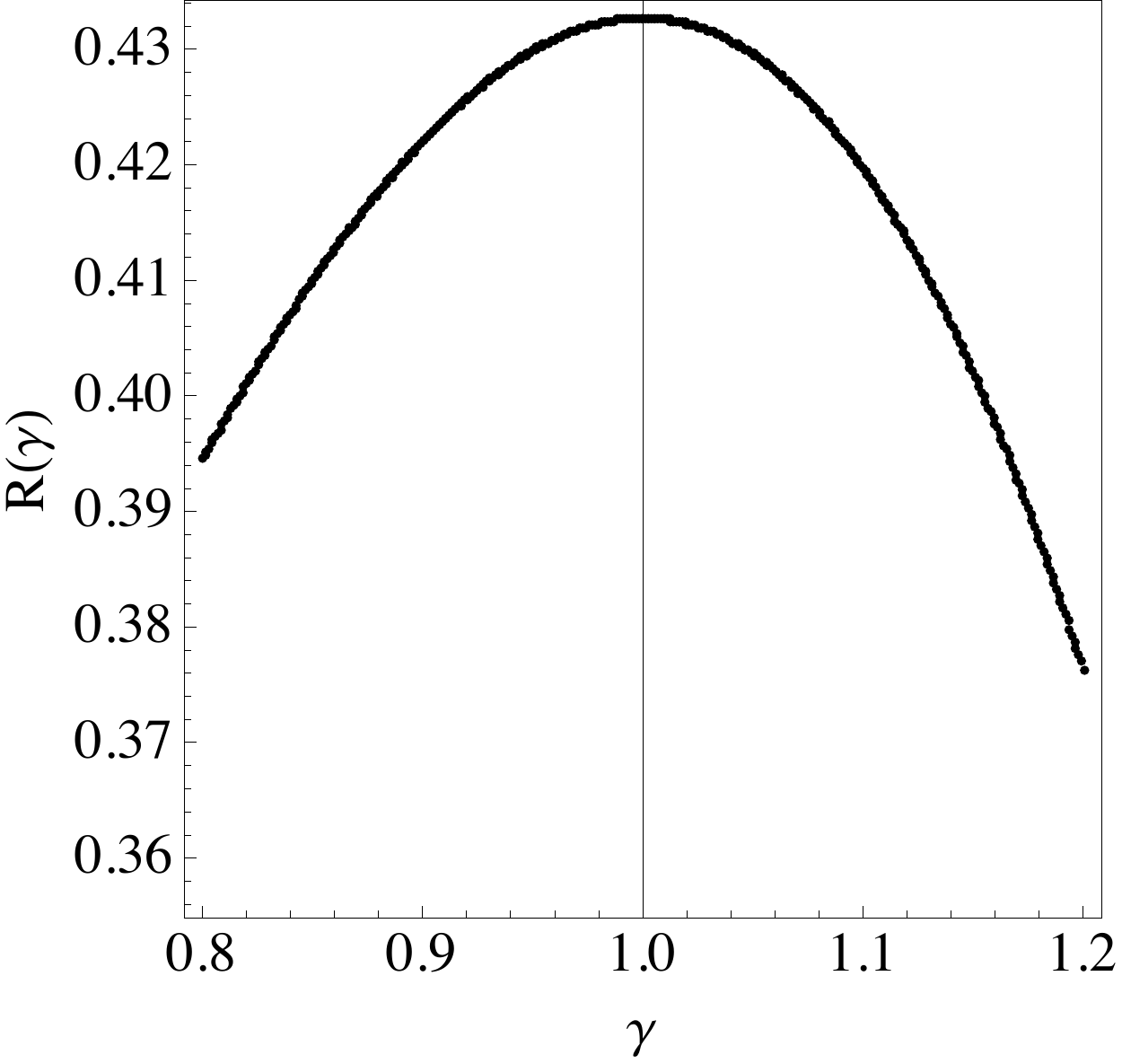}
 \caption{(left) The dispersion relation \eqref{eqn:dispRel} for a quarter wave stack where $b = \frac{n_h}{n_1} \frac{d}{2}$. There is a spectral band gap centered at $\omega = \frac{\pi}{d n_h}$ of width $\frac{4}{d n_h} \sin^{-1}\frac{|n_1 - n_2|}{n_1 + n_2}$. (right) For $n_1=1$ and $n_2=2$, we let $b = \gamma \frac{n_h}{n_1} \frac{d}{2}$ and plot $\gamma$ vs. $R(\gamma)$, the spectral gap to midgap ratio as defined in \eqref{eqn:gapmidgap}. We find that $R(\gamma)$ is maximal precisely when $\gamma=1$, the structure defined by the Bragg relation. }
 \label{fig:bragg}
 \end{center}
 \end{figure}

The Bragg relation given in Eq. \eqref{eqn:Bragg} can be interpreted to state that  constructive interference occurs when the path length of a reflected wave is equal to (a multiple of) the wavelength. In fact, the Bragg structure is optimal in the following sense. For fixed $n_1$, $n_2$ and $d$, let 
\begin{equation}
\label{eqn:bgamma}
b(\gamma) = \gamma  \frac{1}{4} \frac{2 \pi}{\omega n_{1}} = 
\gamma  \frac{1}{4} \lambda_{\text{eff}}, \qquad \gamma \in (0,1)
\end{equation}
define a class of periodic layered structures. The Bragg structure is defined by $b(1)$. Let $\omega_{1}$ and $\omega_{2}$ be the left and right edges of the first spectral bandgap. For a given layered media, define the spectral gap to midgap ratio by 
\begin{equation}
\label{eqn:gapmidgap}
R := \frac{ |\omega_{2} - \omega_{1}| }{ (\omega_{1} + \omega_{2})/2 }
\end{equation}
In Fig. \ref{fig:bragg}(right), we plot the the spectral gap to midgap ratio for the class of devices defined as in \eqref{eqn:bgamma}. The maximal gap to midgap ratio occurs for the Bragg structure ($\gamma = 1$). Lastly, we note that  the Bragg structure does not maximize the spectral gap width, given by $|\omega_{2} - \omega_{1}|$ \cite{OstingThesis}. 

\bibliographystyle{amsalpha}
\bibliography{optRes.bib}
\end{document}